\documentclass[12pt,english]{article}

\usepackage{amssymb}
\usepackage{polski}
\usepackage{ifthen}
\selecthyphenation{polish}
\mathchardef\ogon="012C%
\newcommand{\as}{a\kern-0.22em\lower.40ex\hbox{$_{\ogon}$}}
\newcommand{\As}{A\kern-0.22em\lower.40ex\hbox{$_{\ogon}$}}
\newcommand{\es}{e\kern-0.24em\lower.40ex\hbox{$_{\ogon}$}}
\newcommand{\Es}{E\kern-0.22em\lower.40ex\hbox{$_{\ogon}$}}
\makeatletter

\usepackage[a4paper, left=2.5cm, right=2.5cm, top=2.5cm, bottom=3.5cm, headsep=1.2cm]{geometry}
\newtheorem{theorem}{Theorem}[section]

\newtheorem{corollary}[theorem]{Corollary}

\newtheorem{definition}[theorem]{Definition}
\newtheorem{example}[theorem]{Example}

\newtheorem{lemma}[theorem]{Lemma}

\newtheorem{proposition}[theorem]{Proposition}
\newtheorem{remark}[theorem]{Remark}

\newenvironment{proof}[1][Proof]{\noindent\textbf{#1.} }{\ \rule{0.5em}{0.5em}}
\def\qed{\hbox to 0pt{}\hfill$\rlap{$\sqcap$}\sqcup$}
\makeatother
\usepackage{amssymb}
\usepackage{amsmath}
\usepackage{amsfonts}
\usepackage{color}
\usepackage{graphicx}
\usepackage{graphics}
\numberwithin{equation}{section}
\usepackage{babel}
\date{}
\title{Characterisation of zero duality gap for  optimization problems  in spaces without linear structure}
\author{Ewa Bednarczuk and Monika Syga }

\begin{document}
	\maketitle

\begin{abstract} 
We prove sufficient and necessary conditions ensuring zero duality gap for Lagrangian duality in some classes of nonconvex optimization problems. To this aim, we use the $\Phi$-convexity theory and minimax theorems for $\Phi$-convex functions.  The obtained zero duality results apply to optimization problems involving prox-bounded functions, DC functions, weakly convex functions and paraconvex functions as well as infinite-dimensional linear optimization problems, including Kantorovich duality which plays an important role in determining Wasserstein distance.

\medskip{}

\textbf{Mathematics Subject Classification (2000)32F17; 49J52; 49K27; 49K35; 52A01}
\end{abstract}

\section{Introduction}
\label{intro}

 Zero duality gap conditions have far-reaching consequences for solution methods and theory in convex optimization in finite and infinite dimensional spaces. For the current state-of-the-art of the topic of convex conjugate duality we refer the reader to the monograph by Radu Bo\c{t} \cite{bot}. 

There exist numerous attempts to construct pairs of dual problems in nonconvex optimization e.g., for DC functions \cite{MARTINEZLEGAZ}, \cite{toland}, for composite functions \cite{Bot2003}, DC and composite functions \cite{SunLongLi}, \cite{Sun} and for prox-bounded functions \cite{HareWarPol}.  

 In the present paper, we provide sufficient and necessary conditions for the zero duality gap in Lagrangian duality without appealing to the linearity of the (primal) argument space. Conditions of this kind can pave the way to better understanding duality relationships {\em e.g.} in metric spaces with no linear structure. In this context,  Wasserstein spaces appearing in recent applications see in  Frankowska, data analysis can serve as an example. Riemmannian manifolds  Bacak. 

In the sequel, we distinguish sets of elementary functions, denoted e.g. by $\Phi$, which are defined on a set $X$. $\Phi$-convex functions are pointwise suprema of elementary minorizing functions $\varphi\in\Phi$ (e.g. quadratic, quasi-convex).  This corresponds to the classical fact that proper lower semicontinuous
	convex functions are pointwise suprema of affine minorizing functions.
	$\Phi$-convexity provides a unifying framework for dealing with important classes of nonconvex functions, e.g., paraconvex (weakly convex), DC, and prox-bounded functions.  In the context of duality theory, $\Phi$-convexity is investigated in 
	a large number of papers, e.g.  \cite{ioffe-rub}, \cite{bard}, \cite{dolecki-k},  \cite{Jey2007}, \cite{rolewicz1994}.

	The underlying concepts of $\Phi$-convexity are $\Phi$-conjugation and $\Phi$-subdifferentiation, which
 mimic the corresponding constructions of convex analysis,
	i.e., the $\Phi$-conjugate function and the $\Phi$-subdifferential are defined by replacing, in the respective classical definitions,
	the linear (affine) functions with elementary functions $\varphi\in\Phi$ which may not be affine, in general. This motivates the name {\em Convexity without linearity,} coined for $\Phi$-convexity by Rolewicz \cite{rolewicz1994}.



 In this paper, we propose a quite general and flexible framework which allows us to obtain sufficient and necessary conditions for a zero duality gap in Lagrangian dualities for a large class of (generalized, augmented) Lagrangians. The main tool is the so-called {\em  intersection property} (Definition \ref{def_2} of subsection \ref{minimax_theorem})  which is a purely algebraic concept. In the topological setting, for many classes of elementary functions (satisfying the peaking property), the intersection property is equivalent to the lower semi-continuity of the optimal value function $V$ at $y_{0}$.

Of special interest is the class of $\Phi_{lsc}$-convex functions which embodies many important classes of nonconvex functions whose role increases recently in optimization, e.g. prox-bounded functions \cite{RockWets98}, DC (difference of convex) functions \cite{TuyDC}, weakly convex functions \cite{Vial}, paraconvex functions \cite{Rolewiczpara} and lower semicontinuous convex (in the classical sense) functions. Other interesting classes of elementary functions are listed in Example \eqref{exampleone}.

Within the framework of $\Phi$-convexity, the Lagrange duality has been already investigated on different levels of generality in \cite{burachik}, \cite{gomezvidal}, \cite{penotrub}. 

The main contributions of the paper are as follows. 
\begin{description}
\item [(i)]  Theorem \ref{optgen}  provides necessary and sufficient condition, in the form of {\em the intersection property}  (see  Definition \ref{def_2} and also \cite{Syga2018}) for zero duality gap for the pair of Lagrange dual problems $L_{P}$ and $L_{D}$ with the Lagrangian ${\mathcal L}$ satisfying some $\Phi$-convexity assumptions, where $\Phi$ is any class of elementary functions and the class of multiplier functions $\Psi$ is convex. When the class of multiplier functions $\Psi$ is convex, the equivalence of the intersection property to the $\Psi$-convexity of the optimal value function $V$ is proved in Corollary \ref{cor_theorem4.1}. 
\item [(ii)] In Theorem \ref{theorempeak} we relate condition (i) of Theorem \ref{optgen} to the lower semicontinuity of optimal value function $V$ at $y_0$ when the class of elementary multiplier functions $\Psi$ consists of continuous functions and satisfies the peaking property (Definition \ref{peakingpropertyrolewicz}).
\item [(iii)] Zero duality gap for the multiplier class $\Psi$ defined in Example \ref{example5} is provided in Theorem \ref{theorempeakex} and the proof of Kantorovich duality is given in Theorem \ref{th_kantorovich}.
\end{description}
 The organization of the paper is as follows. In Section 2 we recall the basic concepts of $\Phi$-convexity. We close Section 2 with a  minimax theorem from  \cite{Syga2018} which is the starting point for our investigations. In Section 3 we introduce the Lagrange function and we provide preliminary properties of the Lagrangian duals for optimization problems involving
$\Phi$-convex functions (Proposition \ref{propoptvalue}).

 Section 4 and Section 5, contain the main results of the paper, Theorem \ref{optgen} and Theorem \ref{theorempeak} which provide sufficient and necessary conditions for zero duality gap for primary (argument) classes  $\Phi$ which are convex sets and sets with the peaking property (Definition \ref{peakingpropertyrolewicz}), respectively.  In Section 6 we specialize our zero duality results for particular classes of primal and multiplier elementary functions.

\section{Preliminaries}
\label{preliminaries}

Let $X$ be a set. A function $f:X\rightarrow [-\infty, +\infty ]$ is proper if its  domain
$\text{dom\,}f:=\{x\in X\mid f(x)<+\infty\}\neq\emptyset$ and $f(x)>-\infty$ for all $x\in X$.

Let $\Phi$ be a set of real-valued functions $\varphi:X\rightarrow \mathbb{R}$ closed under the addition of a constant. 

Let $f:X\rightarrow [-\infty, +\infty ]$ be a proper function. 
The set
$$
\text{supp}_{\Phi}(f):=\{\varphi\in \Phi\ :\ \varphi\le f\}
$$
is called the {\em support} of $f$ with  respect to $\Phi$, where, 
for any $g,h:X\rightarrow [-\infty, +\infty ]$,
$g\le h\ \Leftrightarrow\  g(x)\le h(x)\ \ \forall\ x\in X.$
We will use the notation $\text{supp}(f)$ 
whenever the class $\Phi$ is clear from the context. Elements of class $\Phi$ are called elementary functions. A function $f^{\Phi}:X \rightarrow [-\infty, +\infty ]$ defined as
\begin{equation}
	f^{\Phi}(x):=\sup\{\varphi(x)\ :\ \varphi\in\textnormal{supp}(f)\}\ \ \forall\ x\in X.
\end{equation}
is called $\Phi$-convexification of $f$.
\begin{definition}(\cite{dolecki-k}, \cite{rolewicz}, \cite{rubbook})
	\label{convf}
	A function $f:X\rightarrow
	[-\infty, +\infty ]$ is called {\em $\Phi$-convex on $X$} if $f(x)=f^{\Phi}(x)$ for all $x\in X$.
If the set $X$ is clear from the context, we simply say that $f$ is {\em $\Phi$-convex}. 
 
 A function $f:X\rightarrow
	[-\infty, +\infty ]$ is called {\em $\Phi$-convex at $x_{0}\in X$} if
	$$
	f(x_{0})=f^{\Phi}(x_0).
	$$
\end{definition}
A function $f:X\rightarrow
	[-\infty, +\infty ]$ is  $\Phi$-convex on $X$ if $f$ is $\Phi$-convex at any $x_{0}\in X$.
If $\textnormal{supp}(f)=\emptyset$ then  $f^{\Phi}\equiv-\infty$, {\em i.e.,} the function $f\equiv-\infty$ is  $\Phi$-convex (c.f. \cite{rubbook}, section 1.2).

If $X$ is a topological space and a  class $\Phi$ consists of elementary functions $\varphi:X\rightarrow\mathbb{R}$  which are lower semicontinuous on $X$, then $\Phi$-convex functions are lower semicontinuous on $X$ (\cite{zalinescu2002}).
Note that $\Phi$-convex functions as defined above may admit the value $+\infty$   allowing us to consider indicator functions within the framework of $\Phi$-convexity. 

Analogously,  we say that $f:X\rightarrow[-\infty, +\infty ]$ is $\Phi$-concave on $X$ if $-f$ is $-\Phi$-convex on $X$.

The following classes of elementary functions are of interest in applications.

\begin{example}
\label{exampleone}\mbox{}
	
	\begin{enumerate}
	\item \label{example1}
 Let $X$ be a topological vector space,  $X^{*}$ be the dual space to $X$, and
 $$
	\Phi_{conv}:=\{\varphi : X \rightarrow \mathbb{R},\ \varphi(x)= \left\langle \ell,x\right\rangle+c, \ \ x\in X,\  \ell\in X^{*}, \ c\in \mathbb{R}\}, 
	$$
	 It is well-known 
	(see for example Proposition 3.1 of \cite{Ekeland}) that a proper convex lower semicontinuous function $f:X\rightarrow [-\infty,+\infty]$ is  $\Phi_{conv}$-convex. For the analysis of a class of elementary functions which generates all convex functions, see e.g., \cite{Syga2018}.
\item \label{example2}
If $X$ be a Hilbert space, and
\begin{equation} 
\label{philsc}
\Phi_{lsc}:= \{\varphi : X \rightarrow \mathbb{R}, \ \varphi(x)=-a\|x\|^2+ \left\langle \ell,x\right\rangle+c, \ \ x\in X,\  \ell\in X, \ a\geq 0, \ c\in \mathbb{R} \}.
\end{equation}
A function $f: X\rightarrow  (-\infty, +\infty ]$ is $\Phi_{lsc}$-convex iff $f$ is lower semicontinuous and minorized by a quadratic function $q(x):=-a\|x\|^{2}-c$
on $X$ (e.g. \cite{rubbook}, Example 6.2). The class of $\Phi_{lsc}$- convex functions encompass:  prox-bounded functions \cite{poli-rock96} and weakly convex functions  \cite{Vial}, known also under the name paraconvex functions \cite{Rolewiczpara} and semiconvex functions \cite{cannarsa}. 
\item \label{example3}  Let $X$ be a Hilbert space, and
$$
\Phi^{+}_{lsc}:= \{\varphi : X \rightarrow \mathbb{R}, \ \varphi(x)=a\|x\|^2+ \left\langle \ell,x\right\rangle+c, \ \ x\in X,\  \ell\in X^{*}, \ a\geq 0, \ c\in \mathbb{R} \}.
$$ 
 A funcion $f: X\rightarrow  (-\infty, +\infty ]$ is $\Phi^{+}_{lsc}$-concave iff $f$ is upper semicontinuous and majorized by a quadratic function $q(x):=a\|x\|^{2}-c$
on $X$ (\cite{rubbook}, Example 6.2). 
\item \label{example4}
$$
\Phi_{\sigma,\nu}:= \{\varphi : X \rightarrow \mathbb{R}, \ \varphi(x)=a\sigma (x)+ \nu (x)+c, \ \ x\in X, \ a\geq 0, \ c\in \mathbb{R} \},
$$ 
where $X$ is a metric space and $\sigma,\nu:X\rightarrow\mathbb{R}$ are continuous functions such that $\sigma(0)=\nu(0)=0$, see \cite{Huang} and \cite{RubHua}.
 \item \label{example5}
 Let $(X,d)$ be a metric space, and
\begin{equation}
\label{phid}
    \Phi_{d}:= \{\varphi : X \rightarrow \mathbb{R}, \ \varphi(x)=-a d(x,x_{0})+c, \ \ x_{0}\in X,\  \ a>0, \ c\in \mathbb{R} \}.
\end{equation}
 A function $f:X\rightarrow\mathbb{R}$ is globally Lipschitz if and only if $f$ is $\Phi_{d}$-convex. (\cite{rolewicz}, Proposition 2.1.6).  
 \end{enumerate}
\end{example}

\subsection{$\Phi$-conjugation}
\label{conjugation}

Let $f:X\rightarrow(-\infty,+\infty]$. The function $f^{*}_{\Phi}:\Phi\rightarrow(-\infty,+\infty]$,
\begin{equation}
\label{conjugate} 
f^{*}_{\Phi}(\varphi):=\sup_{x\in X} (\varphi(x)-f(x))
\end{equation}
is called the {\em $\Phi$-conjugate} of $f$. The function $f^{*}_{\Phi}$ is $\Phi$-convex (c.f. Proposition 1.2.3 of \cite{rolewicz}). Accordingly, the {\em second $\Phi$-conjugate} of $f$, $f^{**}_{\Phi}:X\rightarrow(-\infty,+\infty]$, is defined as
$$
f^{**}_{\Phi}(x):=\sup_{\varphi\in\Phi}(\varphi(x)-f^{*}_{\Phi}(\varphi)).
$$

\begin{theorem} (\cite{rolewicz} and Theorem 1.2.6, \cite{rubbook}, Theorem 7.1) 
	\label{conju}
	Function $f:X\rightarrow(-\infty,+\infty]$ is $\Phi$-convex if and only if 
		$$
		f(x)=f^{**}_{\Phi}(x) \ \ \ \ \ \ \ \ \ \ \ \forall \ \ x\in X.
		$$	
\end{theorem}

\subsection{Minimax Theorem }
\label{minimax_theorem}
The main tool used in   proving the zero duality gap for Lagrangean duality (introduced in  section \ref{section_2}) is Theorem \ref{new_min_max}, a minimax theorem for $\Phi$-convex functions. 

The crucial ingredient of this theorem is the intersection property, a necessary and sufficient condition for the minimax equality to hold,
introduced in \cite{bed-syg} and investigated in \cite{Syga2018}, \cite{sygamoor}.
\begin{definition}
	\label{def_2}
	 Let $X$ be a set. Let $\varphi_{1},\varphi_{2}:X\rightarrow\mathbb{R}$ be any two functions and $\alpha\in\mathbb{R}$. We say that $\varphi_{1}$ and $\varphi_{2}$ have
	{\em the intersection property on $X$ at the level $\alpha\in \mathbb{R}$} 
	iff  for every $t\in [0,1]$
	{\small 
		\begin{equation}
		\label{eq-n}
		\begin{array}[t]{c}
		[t\varphi_{1}+(1-t)\varphi_{2}<\alpha]\cap [\varphi_{1}<\alpha]=\emptyset \ \ \ \text{or}
		\ \ \
		[t\varphi_{1}+(1-t)\varphi_{2}<\alpha]\cap [\varphi_{2}<\alpha]=\emptyset,
		\end{array}
		\end{equation}}
		where $[\varphi<\alpha]:=\{x\in X :\ \varphi(x)<\alpha \}$.
\end{definition}

 It follows directly from the definition that it must be $\varphi_{1}\neq\varphi_{2}$.
 By Proposition 4 of \cite{bed-syg}, when $\varphi_{1},\varphi_{2}$ are taken from  the classes $\Phi_{conv}$, $\Phi_{lsc}$ or $\Phi_{lsc}^{+}$ the condition \eqref{eq-n} can be equivalently rewritten as
\begin{equation}
\label{eq-n1}
    [\varphi_{1}<\alpha]\cap[\varphi_{2}<\alpha] =\emptyset.
\end{equation}
This is not true for classes $\Phi_{\sigma,\nu}$ and $\Phi_{d}$.

The proof of the general minimax theorem relies on the following key lemma.
 
\begin{lemma}(Lemma 4.1, \cite{Syga2018})
\label{key_lemma}
Let $X$ be a set, $\alpha\in\mathbb{R}$, and let $\varphi_{1},\varphi_{2}:X\rightarrow\mathbb{R}$ be any two functions.  The functions $\varphi_{1}$ and $\varphi_{2}$ have the intersection property on $X$ at the level $\alpha$ if and only if $\exists\ t_{0}\in[0,1]$ such that
\begin{equation}
    \label{key_ineq} 
    t_{0}\varphi_{1}+(1-t_{0})\varphi_{2}\ge \alpha\ \ \forall\ x\in X.
\end{equation}
\end{lemma}
 Lemma 4.1  proved in \cite{Syga2018} refers to a more general situation, where $\varphi_{1},\ \varphi_{2}:X\rightarrow [-\infty,+\infty]$ and reduces to Lemma \ref{key_lemma} whenever $\varphi_{1},\ \varphi_{2}:X\rightarrow \mathbb{R}$.

By applying  Lemma \ref{key_lemma} we get the following minimax theorem (which was proved in a more general situation in \cite{Syga2018}).
\begin{theorem} (\cite{Syga2018}, Theorem 5.2)
	\label{new_min_max}
	Let $X$ be a nonempty set and $Z$ be a convex subset of a real vector space U. Let $\Phi$ be a class of elementary functions $\varphi:X\rightarrow\mathbb{R}$.  Let $a:X\times Z\rightarrow[-\infty,+\infty]$ be a function such that for any $x\in X$ the function $a(x,\cdot):Z\rightarrow[-\infty,+\infty]$   is concave on $Z$ and for any $z\in Z$ the function $a(\cdot,z):X\rightarrow[-\infty,+\infty]$ is $\Phi$-convex on $X$. 
	
The following conditions are equivalent:
\begin{description}
	\item [{\em (i)}] for every $\alpha\in\mathbb{R}$, $\alpha < \inf\limits_{x\in X} \sup\limits_{z\in Z} a(x,z)$, there exist $z_{1}, z_{2}\in Z$ and $\varphi_{1}\in \textnormal{supp}_{\Phi} a(\cdot, z_{1})$, $\varphi_{2}\in \textnormal{supp}_{\Phi} a(\cdot, z_{2})$ such that the intersection property holds for $\varphi_{1}$ and $\varphi_{2}$ on $X$ at the level $\alpha$,
	\item [{\em (ii)}] $\sup\limits_{z\in Z} \inf\limits_{x\in X} a(x,z)=\inf\limits_{x\in X} \sup\limits_{z\in Z} a(x,z).$
\end{description}
\end{theorem}
\begin{proof}
 Let $\alpha <\inf\limits_{x\in X} \sup\limits_{z\in Z} a(x,z)$. By $(i)$, there exist $z_1,z_2\in Z$ and  $ \bar{\varphi}_1 \in\text{supp} a(\cdot,z_1)$  and $ \bar{\varphi}_2 \in\text{supp}a(\cdot,z_2)$ such that $ \bar{\varphi}_1 \in\text{supp} a(\cdot,z_1)$  and $ \bar{\varphi}_2 \in\text{supp} a(\cdot,z_2)$ have the intersection property on $X$ at the level $\alpha$. By  Lemma \ref{key_lemma} and \eqref{key_ineq}, there exists $t\in [0,1]$ such that
 \begin{equation}
    \label{key_ineq1} 
    t\bar{\varphi}_{1}+(1-t)\bar{\varphi}_{2}\ge \alpha\ \ \forall\ x\in X.
\end{equation}

By the definition of the support set and the inequality \eqref{key_ineq1} we get
\begin{equation}
    \label{key_ineq2} 
    ta(\cdot,z_1)+(1-t)a(\cdot,z_2)\ge \alpha\ \ \forall\ x\in X.
\end{equation}
By the concavity of $a$ as a function of $z$, we have
\begin{equation}
    \label{key_ineq3} 
    a(x,z_0)\ge \alpha\ \ \forall\ x\in X,
\end{equation}
where $z_0=tz_1+(1-t)z_2$.
From this, we deduce the following inequality
\begin{equation}
    \label{key_ineq4} 
 \sup\limits_{z\in Z}  \inf\limits_{x\in X} a(x,z)\ge \alpha\ \ \forall\ x\in X.
\end{equation}
By the fact that the inequality  \eqref{key_ineq4} holds for every $\alpha <\inf\limits_{x\in X} \sup\limits_{z\in Z} a(x,z)$  we get the required conclusion.

The converse implication follows directly from Theorem 2.1 of \cite{Syga2018}.
\end{proof}

 Condition $(i)$ of Theorem \ref{new_min_max}
will be referred to as {\em the intersection property} for the function $a(\cdot,\cdot)$ for elementary function classes $\Psi$ and $\Phi$.
\begin{remark}
    Let $X,Z,\Phi$ and $a$ be as in Theorem \ref{new_min_max}.
    \begin{enumerate}
        \item If there exist $\bar{z}\in Z$ and $\bar{x}\in X$ such that $a(\bar{x},\bar{z})=+\infty$ then $a(\bar{x},\cdot)$ is not proper concave function
        \item If there exists $\bar{x}\in X$ such that $a(\bar{x},\cdot)\equiv +\infty$ then $\inf\limits_{x\in X} \sup\limits_{z\in Z} a(x,z)=+\infty$  and for the minimax equality  to hold the intersection property must hold for all $\alpha\in \mathbb{R}$
       \item  If there exists $\bar{z}\in Z$ such that $a(\cdot,\bar{z} )\equiv+\infty$, then $\inf\limits_{x\in X} a(x,\bar{z})=+\infty$, hence
        $\sup\limits_{z\in Z} \inf\limits_{x\in X} a(x,z)=+\infty$. On the other hand $\sup\limits_{z\in Z}a(x, z) =+\infty$, for every $x\in X$, so $\inf\limits_{x\in X} \sup\limits_{z\in Z} a(x,z)=+\infty$ and minimax equality always holds for function $a(\cdot,\cdot)$.
        
        \item If there exists $\bar{x}\in X$ such that $a(\bar{x},\cdot)=-\infty$, then $\inf\limits_{x\in X} \sup\limits_{z\in Z} a(x,z)=-\infty$, and condition (i) of Theorem \ref{new_min_max} always holds.
 \end{enumerate}
\end{remark}

\section{ Lagrangian duality}
\label{section_2}

To introduce the Lagrange function for  problem \eqref{problem} we apply {\em perturbation/parametrization approach}, see e.g. \cite{balder}, \cite{Bonnans},  \cite{toland}.

 Let  $X$ (set of arguments) and $Y$ (set of parameters) be nonempty sets and  $y_0\in Y$. 
A   function $p:X\times Y\rightarrow[-\infty,+\infty]$  
is called {\em a perturbation/parametrization  function to problem $(P_{0})$},
\begin{equation}
\label{problem}
\tag{$P_{0}$}
\text{Min}\ \ \ \ p(x,y_{0}) \ \ \ \ \ \ x\in X.
\end{equation}

 Our standing assumption is that $p$ is proper for any $y\in Y$. ???

Perturbation  function $p(\cdot,\cdot)$, defines the  family of parametric problems $(P_{y})$,
\begin{equation}
\label{problemy}
\tag{$P_y$}
\text{Min}\ \ \ \ p(x,y),\ \ \ \ x\in X,
\end{equation}
where $(P_{y_{0}})$ coincides with $(P_{0})$. 

Let $\Psi$ be a class of elementary functions defined on the set $Y$, $\psi:Y\rightarrow \mathbb{R}$. Class $\Psi$ (e.g. any of the classes from Example \ref{exampleone} defined on the parameter set $Y$), can be seen as a counterpart of sets of Lagrange multipliers.

The  Lagrangian ${\mathcal L}:X\times\Psi:\rightarrow [-\infty,+\infty]$  
is defined as
\begin{equation}
\label{genlag}
{\mathcal L}(x, \psi):=\psi(y_0)-p^{*}_{x}(\psi),
\end{equation}
where, for any fixed $x\in X$, the function $p_{x}^{*}:\Psi\rightarrow (-\infty,+\infty]$,  is the (partial) $\Psi$-conjugate of $p$ with respect to $y$ and 
\begin{equation} 
\label{eq_partial_conjugate}
p^{*}_{x}(\psi)=\sup\limits_{y\in Y}\{\psi(y)-p(x,y)\},
\end{equation}
i.e. $p^{*}_{x}(\cdot)$ is the $\Psi$-conjugate of the functions $p_{x}(\cdot):=p(x,\cdot)$, $x\in X$, (subsection \ref{conjugation}).  When $Z_{1}:=Y$, $Z_{2}:=\Psi$, $c(y,\psi):=\psi(y)$, Lagrangian defined by \eqref{genlag} coincides with Lagrangian $L(x,y)$ as defined in  Proposition 1 of \cite{penotrub}. Also, the Lagrangian given by formula 2.1 of \cite{BurachikRubinov} coincides with \eqref{genlag}.

 Analogous definitions of Lagrangian have been investigated, in convex case {\em e.g.} by  \cite{Bonnans}, in DC case by  \cite{toland} and, in general, abstract convex case \cite{BurachikRubinov}, by \cite{dolecki-k}, \cite{Kurcyusz} and \cite{rolewicz}, Section 1.7. 

 Within the framework of abstract convexity, we provide sufficient and necessary conditions for zero duality gap for the following pair of dual problems.

The  Lagrangian primal problem  is defined as
\begin{equation}
\label{lprob}
\tag{$L_P$}
val(L_{P}):=\inf_{x\in X}\sup_{\psi\in\Psi}{\mathcal L}(x, \psi).
\end{equation}
\begin{proposition}
	\label{primal}
  Problems \eqref{problem} and \eqref{lprob} are equivalent in the sense that 
	$$\inf\limits_{x\in X}p(x,y_0)=\inf\limits_{x\in X}\sup\limits_{\psi\in\Psi}{\cal L}(x,\psi),
	$$
	if and only if  $p(x,\cdot)$ is $\Psi$-convex function at $y_0$ for all $x\in X$. 
\end{proposition}
\begin{proof}
	It is enough to observe that the following equality holds
	$$
	\sup_{\psi\in\Psi}{\mathcal L}(x, \psi)= \sup_{\psi\in\Psi}\{\psi(y_0)-p^{*}_{x}(\psi)\}=p^{**}_{x}(y_0)=p(x,y_0).
 $$
	where the latter equality follows from Theorem \ref{conju}.
\end{proof}

The Lagrangian dual problem to $\eqref{lprob}$ is defined as
\begin{equation}
\label{ldual}
\tag{$L_D$}
val(L_{D} ):=\sup_{\psi\in\Psi}\inf_{x\in X}{\mathcal L}(x, \psi).
\end{equation}
Problem \eqref{ldual} is called the Lagrangian dual.  The inequality
\begin{equation} 
\label{minimaxineq}
val(L_{D})\le val(L_{P})
\end{equation}
always holds. 
We say that 
the {\em zero duality gap} holds for problems \eqref{lprob} and \eqref{ldual} if the equality $val(L_{P})=val(L_{D})$ holds.
\subsection{The $\Psi$-convexity of optimal value function}



The optimal value function of $(P_{y})$,
		$V:Y\rightarrow(-\infty,+\infty]$ is defined as 
		\begin{equation}
		\label{opvalue}
		    	V(y):=\inf_{x\in X}p(x,y).
		\end{equation}

		\begin{proposition}
		\label{propoptvalue}
 Assume that  $p_{x}=p(x,\cdot)$ is $\Psi$-convex function on $Y$ for all $x \in X$.	The following are equivalent:
\begin{description}
	\item[(i)]  	$$\inf_{x\in X}p(\cdot,y_0)=\inf_{x\in X}\sup_{\psi\in\Psi}{\mathcal L}(x, \psi)=\sup_{\psi\in\Psi}\inf_{x\in X}{\mathcal L}(x, \psi).$$
	\item[(ii)] $$
	V(y_0)=V^{**}(y_0),$$
	where $y_0\in Y$.
 \item[(iii)]  $V$ is $\Psi$-convex at $y_0$ i.e. $V(y_0)=\sup\{\psi(y_0), \ \psi\in\text{supp}\, V(y) \}$  (cf. Theorem \ref{conju}).
\end{description}		
			\end{proposition}
			\begin{proof}
			 	For any $\psi\in \Psi$, we have
	\begin{equation}
	\label{con2}
		V^*(\psi)=\sup_{y\in Y}\{\psi(y)-  \inf_{x\in X}p(x,y)\}=\sup_{y\in Y}\sup_{x\in X}\{\psi(y)-  p(x,y)\}=\sup_{x\in X}p^*_{x}(\psi).
		\end{equation}
		On the other hand, 
		$$
		\inf_{x\in X}{\mathcal L}(x,\psi)=	\inf_{x\in X}\{ \psi(y_0)-p^*_{x}(\psi)\}= \psi(y_0)-\sup_{x\in X}\{p^*_{x}(\psi)\}.
		$$
		By \eqref{con2},
		$\inf_{x\in X}{\mathcal L}(x,\psi)=\psi(y_0)-V^*(\psi)$,
		and for  the dual \eqref{ldual} we have
		\begin{equation} 
		\label{dualvalue}
		\sup_{\psi\in\Psi}\inf_{x\in X}{\mathcal L}(x, \psi)=	\sup_{\psi\in\Psi}\{\psi(y_0)-V^*(\psi)\} =V^{**}(y_0).
		\end{equation}
		By Proposition \ref{primal}, when the perturbation function $p(x,\cdot)$ is $\Psi$-convex for each $x\in X$ then
		$$
		\inf_{x\in X}\sup_{\psi\in\Psi}{\mathcal L}(x, \psi)=V(y_0)
		$$
		which completes the proof.
			\end{proof}
			
	\begin{remark}
 \begin{description}
     \item[(a)] The following equivalence \begin{description}
	\item[(i)]  	$$\inf_{x\in X}p(\cdot,y_0)=\sup_{\psi\in\Psi}\inf_{x\in X}{\mathcal L}(x, \psi).$$
	\item[(ii)] $$
	V(y_0)=V^{**}(y_0),$$
	where $y_0\in Y$.
\end{description}		
			has been investigated in \cite{RubHua}.

		    \item[(b)] If, 
		    $$\inf_{x\in X}p(\cdot,y_0)=\sup_{\psi\in\Psi}\inf_{x\in X}{\mathcal L}(x, \psi),$$
		    then
		    	$$\inf_{x\in X}p(\cdot,y_0)=\inf_{x\in X}\sup_{\psi\in\Psi}{\mathcal L}(x, \psi)=\sup_{\psi\in\Psi}\inf_{x\in X}{\mathcal L}(x, \psi).$$
		    By inequalities
		    $$
		    val(L_{D})\le val(L_{P}) \ \ \ \ \ p(x,y_0)\geq p^{**}(x,y_0)
		    $$
		    we have\\
      
		    $\inf\limits_{x\in X}p(\cdot,y_0)=\sup\limits_{\psi\in\Psi}\inf\limits_{x\in X}{\mathcal L}(x, \psi)\leq \inf\limits_{x\in X}\sup\limits_{\psi\in\Psi}{\mathcal L}(x, \psi)=p^{**}(x,y_0) \leq p(x,y_0)= \inf\limits_{x\in X} p(\cdot,y_0)$
		    \\ 
            i.e
		    $$\inf_{x\in X}p(\cdot,y_0)=\inf_{x\in X}\sup_{\psi\in\Psi}{\mathcal L}(x, \psi)$$
		    \end{description}
		\end{remark}
			
			In reflexive Banach spaces for some particular elementary functions, conditions ensuring $(ii)$ were proved in  Theorem 4.1 and Proposition 4.1 of \cite{BurachikRubinov}.

\section{Zero duality gap via intersection property}
We  start with some preliminary observations. We have
$$
\text{dom\,} {\mathcal L}:\begin{array}[t]{l}
=\{(x,\psi)\in X\times \Psi\mid {\mathcal L}(x,\psi)<+\infty\}\\
=\{(x,\psi)\in X\times \Psi\mid \inf\limits_{y\in Y}\{p_{x}(y)-\psi(y)\}<+\infty\}.
\end{array}
$$
Observe that for any $\psi\in\Psi$,
$$
p^{*}_{x}(\psi)\ge \psi(y_{0})-p(x,y_{0})
$$
 i.e. $p^{*}_{x}(\cdot)>-\infty$ for any $x\in\text{dom\,}p(\cdot,y_0)$. Since $p(\cdot,y_0)$ is  proper,  $\text{dom\,}p(\cdot,y_0)\neq\emptyset$  and
$\text{dom\,}p(\cdot,y_0)\subset\text{dom\,}{\mathcal L}(\cdot,\psi)$ 
 for any $\psi\in\Psi$.

On the other hand, under the assumptions of Proposition \ref{primal}, since $p(\cdot,y_0)$ is proper we have $p(\cdot,y_0)=\sup\limits_{\psi\in\Psi}{\mathcal L}(x,\psi)>-\infty$ for any $x\in X$, i.e. ${\mathcal L}(x,\psi)>-\infty$ for some $\psi\in\Psi$ which means that among functions ${\mathcal L}(\cdot, \psi)$, $\psi\in\Psi$ may exist proper functions. In conclusion,
\begin{equation} 
\label{proper}
\text{dom\,}{\mathcal L}(\cdot,\psi) \neq\emptyset\text{  for every  } \psi\in\Psi\ \ \text{and}\ \text{supp\,} {\mathcal L}(\cdot,\psi)\neq\emptyset\ \ \ \text{for some  } \psi\in\Psi.
\end{equation}
 Moreover, the following fact holds. 
 \begin{equation}
 \label{observation}
 \exists_{\bar{x}\in X}\ \exists_{\bar{\psi}\in\Psi} {\mathcal L}(\bar{x},\bar{\psi})=+\infty \ \Leftrightarrow\ \forall_ {\psi\in\Psi}\ {\mathcal L}(\bar{x},\psi)=+\infty.
 \end{equation}

 To see this it is enough to note that the condition
 ${\mathcal L}(\bar{x},\bar{\psi})=\bar{\psi}(y_{0})-p^{*}_{\bar{x}}(\bar{\psi})=+\infty$ for some $\bar{x}\in X$ and $\bar{\psi}\in\Psi$ can be  rewritten as
 $\inf_{y\in Y}\{p_{\bar{x}}(y)-\bar{\psi}(y)\}=+\infty$ 
which means that $p_{\bar{x}}(y)=p(\bar{x},y)=+\infty$ for any $y\in Y$ i.e.
 $\text{dom\,} p_{\bar{x}}=\emptyset$ and consequently $\forall_ {\psi\in\Psi}\ {\mathcal L}(\bar{x},\psi)=+\infty$.

Let $X$ be a nonempty set (primal argument set) and let $\Phi$ be a class of elementary functions defined on $X$,  $\varphi:X\rightarrow \mathbb{R}$. 
The following theorem provides sufficient and necessary conditions for zero duality gap, 
within the framework of abstract convexity. To our knowledge, this is the first result on this level of generality.
\begin{theorem}
	\label{optgen}
	Let $X$ and $Y$ be nonempty sets. Let $\Phi$ be a class of elementary functions, $\phi:X\rightarrow \mathbb{R}$. Let $\Psi$ be a convex set of elementary functions defined on $Y$,  $\psi:Y\rightarrow \mathbb{R}$ and let the  function ${\mathcal L}(\cdot,\psi):X\rightarrow[-\infty,+\infty]$, given by \eqref{genlag},
	be $\Phi$-convex on $X$   for any $\psi\in\Psi$. Assume that $p(x,\cdot)$ is $\Psi$-convex function at $y_{0}\in Y$ for all $x\in X$.

	The following are equivalent:
	\begin{description}
\item[(i)]	for every $\alpha <\inf\limits_{x\in X}\sup\limits_{\psi\in\Psi}{\mathcal L}(x, \psi)$ there exist $\psi_1,\psi_2\in \Psi$ and  $ \varphi_1 \in\text{supp} {\mathcal L}(\cdot,\psi_1)$  and $ \varphi_2 \in\text{supp} {\mathcal L}(\cdot,\psi_2)$ such that there exists $t_0\in[0,1]$ such that
$$
t_0\varphi_1(x)+(1-t_0)\varphi_2(x)\geq \alpha \ \ \ \ \ \forall \ \ x\in X.
$$

\item[(ii)]	$$\inf\limits_{x\in X} p(\cdot,y_0)=\inf_{x\in X}\sup_{\psi\in\Psi}{\mathcal L}(x, \psi)=\sup_{\psi\in\Psi}\inf_{x\in X}{\mathcal L}(x, \psi).$$
\end{description}
\end{theorem}
\begin{proof}
To prove the equivalence we use Theorem  \ref{new_min_max}. First we show that ${\mathcal L}(x,\cdot)$ is a concave function of $\psi$ for all $x\in X$, {\em i.e.,} 
\begin{equation} 
\label{concavity} 
{\mathcal L}(x, t\psi_1+(1-t)\psi_2)\ge t{\mathcal L}(x, \psi_1)+(1-t){\mathcal L}(x, \psi_2)
\end{equation}
for any $t\in[0,1]$, $x\in X$ and any $\psi_{1},\psi_{2}\in\Psi$.

 Take any  $\psi_{1}, \psi_{2}\in\Psi$. 
${\mathcal L}(x,\psi_{1})$ and ${\mathcal L}(x,\psi_{2})$ are finite.
Let  $t\in [0,1]$. We have
$$
	\begin{array}{l}
{\mathcal L}(x, t\psi_1+(1-t)\psi_2)=\\
=t\psi_1(y_0)+(1-t)\psi_2(y_0)-\sup\limits_{y\in Y}\{t\psi_1(y)+(1-t)\psi_2(y)-p(x,y)\}\\
=t\psi_1(y_0)+(1-t)\psi_2(y_0)-\sup\limits_{y\in Y}\{t\psi_1(y)+(1-t)\psi_2(y)-tp(x,y)-(1-t)p(x,y)\}\\
\geq t\psi_1(y_0)+(1-t)\psi_2(y_0) -t\sup\limits_{y\in Y}\{\psi_1(y)-p(x,y)\}-(1-t)\sup\limits_{y\in Y}\{\psi_2(y)-p(x,y)\}\\
=t{\mathcal L}(x, \psi_1)+(1-t){\mathcal L}(x, \psi_2).
\end{array}
$$
Hence, ${\mathcal L}(x,\cdot)$ is concave on $\Psi$. 
Now we proceed to prove that $\Leftrightarrow$.

$(i) \Rightarrow (ii)$ All assumptions of Theorem  \ref{new_min_max} hold for the sets $X, Y, \Phi, \Psi$ and the functions ${\mathcal L}(\cdot, \cdot)$, hence from $(i)$ and Theorem  \ref{new_min_max} we have
\begin{equation}
\label{mmeq1}
\inf_{x\in X}\sup_{\psi\in\Psi}{\mathcal L}(x, \psi)=\sup_{\psi\in\Psi}\inf_{x\in X}{\mathcal L}(x, \psi).
\end{equation}
From Proposition \ref{primal} we get 
$$
\inf\limits_{x\in X}p(x,y_0)=\inf\limits_{x\in X}\sup\limits_{\psi\in\Psi}{\cal L}(x,\psi),
$$
which, together with \eqref{mmeq1}, gave $(ii)$.

$(ii) \Rightarrow (i)$ Follows directly from Theorem \ref{new_min_max}.

	\end{proof}

\begin{remark}
    \label{remarkinfinite} 
    Observe that ${\mathcal L}(x,\psi_{1})$ and ${\mathcal L}(x,\psi_{2})$ can take infinite values.
    Assume that  $p(x,\cdot)$ is $\Psi$-convex function at $y_{0}\in Y$ for all $X$. By Proposition \ref{primal}, $p(x,y_0)=\sup\limits_{\psi\in\Psi}{\mathcal L}(x,\psi)$. Since  $p(\cdot,y_0)$ is proper, it may only happen that
    $$
    \inf\limits_{x\in X}p(\cdot,y_0)=\inf\limits_{x\in X}\sup\limits_{\psi\in\Psi}{\mathcal L}(x,\psi)=-\infty\ \ \text{or   }
     \inf\limits_{x\in X}p(\cdot,y_0)=\inf\limits_{x\in X}\sup\limits_{\psi\in\Psi}{\mathcal L}(x,\psi)\ \ \text{is finite}.
     $$ 
     In the first case, when $val(L_{P})=-\infty$, in view of \eqref{minimaxineq}, there is nothing to prove and the condition $(i)$ of Theorem \ref{optgen} is automatically satisfied.
\end{remark}
The following corollary shows the relationship between the $\Phi$-convexity of the Lagrangian (with respect to $x$) and the $\Psi$-convexity of the optimal value function. {\color{red} ??}
			\begin{corollary}
   \label{cor_theorem4.1}
			Let $\Psi$ be a convex set of elementary functions $\psi:Y\rightarrow \mathbb{R}$. Assume that  $p_{x}=p(x,\cdot)$ is $\Psi$-convex function on $Y$ for all $x \in X$.
			Assume that for any $\psi\in\Psi$ the  function ${\mathcal L}(\cdot,\psi):X\rightarrow[-\infty,+\infty]$,  defined by \eqref{genlag},
				is $\Phi$-convex on $X$.  The following are equivalent.
			\begin{description}	
				\item [(i)] For every $\alpha <\inf\limits_{x\in X}\sup\limits_{\psi\in\Psi}{\mathcal L}(x, \psi)$ there exist $\psi_1,\psi_2\in \Psi$ and  $ \varphi_1 \in\text{supp}\, {\mathcal L}(\cdot,\psi_1)$  and $ \varphi_2 \in\text{supp}\, {\mathcal L}(\cdot,\psi_2)$ such that functions $\varphi_1$ and $\varphi_2$ have the intersection property at the level $\alpha$
                \item[(ii)] For every $\alpha <\inf\limits_{x\in X}\sup\limits_{\psi\in\Psi}{\mathcal L}(x, \psi)$ there exist $\psi_1,\psi_2\in \Psi$ and  $ \varphi_1 \in\text{supp}\, {\mathcal L}(\cdot,\psi_1)$  and $ \varphi_2 \in\text{supp}\, {\mathcal L}(\cdot,\psi_2)$ and $t_0\in [0,1]$ such that 
$$
    t_{0}\varphi_{1}+(1-t_{0})\varphi_{2}\ge \alpha\ \ \forall\ x\in X.
$$
				\item [(iii)] The optimal value function  $V$ is $\Psi$-convex at  $y_0$. 
				\end{description}
				\end{corollary}
				\begin{proof} 
				Follows from Theorem 	\ref{optgen}, Proposition \ref{propoptvalue} and Lemma \ref{key_lemma}.   To see the implication $(i)$ $\Rightarrow$ $(iii)$ we need to show the following. 
				\begin{description} 
				\item[(a)] If $V(y_{0})=\inf\limits_{x\in X} f(x)=-\infty$, then $V(y)=-\infty$ for all $y\in Y$.
				\item[(b)] If $V(y_{0})$ is finite, then $\text{supp\,}\neq\emptyset$ and 
				\begin{equation} 
				\label{optimalvalue} 
				V(y)=\sup\{\psi(y)\mid \psi\in\text{supp\,} V\}.
				\end{equation}
				\end{description}
				
				Ad $(a)$. In this case, $\text{supp\,} V=\emptyset$, i.e. $V\equiv-\infty$.
				
				Ad $(b)$. Assume that $V(y_{0})=\inf\limits_{x\in X}\sup\limits_{\psi\in\Psi}{\mathcal L}(x,\psi)$. By $(i)$, in view of Theorem 	\ref{optgen} and Proposition \ref{propoptvalue}, we have $V^{**}(y_{0})=\sup\limits_{\psi\in\Psi}\inf\limits_{x\in X}{\mathcal L}(x,\psi)=V(y_{0})$. By Theorem \ref{conju}, $V$ is $\Psi$-convex at $y_{0}$.
				\end{proof}

\section{Intersection property and the lower semicontinuity of optimal value function}

In the zero-duality theorem, Theorem \ref{optgen}, we assumed that  $\Psi$ (the multipliers) is a convex set. There are important examples of classes for which this convexity assumption does not hold (see e.g., Example 2.2 (\ref{example5})). In this section, we use the so-called peaking property (Definition \ref{peakingpropertyrolewicz}) to prove zero duality gap for classes $\Psi$ which are not convex (Theorem 5.3 below).

\begin{definition}(section 2.1, \cite{rol-glob})
    \label{peakingpropertyrolewicz}
    Let $(Y,d_Y)$ be a metric space. A family $\cal{G}$ has a {\em peaking property at $y_0\in Y$} if for every positive numbers $\varepsilon, \delta, K$ every $g\in {\cal G}$ there exists function $\bar{g}\in {\cal G}$ such that for all $y\in Y$
    $$
    \bar{g}(y)\leq \varepsilon    $$
and if $d_Y(y,y_0)\geq \delta $ then

     $$
    \bar{g}(y)\leq g(x)-K.    $$
    \end{definition}
The notion of peaking property is strictly related to a sharp class of functions introduced by Lindberg. The notion of sharpness is more general than needle classes of functions introduced by Balder (\cite{balder}).
\begin{proposition}(Proposition 2.1.2, \cite{rol-glob})
\label{rol-pro}
     Let $(Y,d_Y)$ be a metric space. Assume that a family $\cal{G}$ consists of continuous functions $g:Y\rightarrow\mathbb{R}$ and has the peaking property at $y_0\in Y$. A function $h:Y\rightarrow [-
     \infty,+\infty]$ is ${\cal G}$-convex at $y_0\in Y$ if and only if
     \begin{itemize}
         \item[(i)] there exists $\tilde{g}\in {\cal G}$ such that $h\geq \tilde{g}$
         \item[(ii)] $h$ is lower semicontinuous at $y_0\in Y$.
     \end{itemize}
\end{proposition}

In the theorem below we relate condition (i) of Theorem \ref{optgen} to the lower semicontinuity of optimal value function $V$ at $y_0$ when the class of elementary multiplier functions $\Phi$ consists of continuous functions and satisfies the peaking property.
\begin{theorem}
\label{theorempeak}
		Let $X$ be a nonempty set, and $(Y,d)$ be a metric space. Let $\Psi$ be a set of elementary continuous functions $\psi:Y\rightarrow \mathbb{R}$ with the peaking property at $y_0$ and let the  Lagragian  ${\mathcal L}(\cdot,\psi):X\rightarrow[-\infty,+\infty]$, given by \eqref{genlag},
	be $\Phi$-convex on $X$   for any $\psi\in\Psi$. Assume that $p(x,\cdot)$ is $\Psi$-convex function at $y_{0}\in Y$ for all $x\in X$.
	The following are equivalent:
	\begin{description}
\item[(i)]	for every $\alpha <\inf\limits_{x\in X}\sup\limits_{\psi\in\Psi}{\mathcal L}(x, \psi)$ there exist $\psi_1,\psi_2\in \Psi$ and  $ \varphi_1 \in\text{supp} {\mathcal L}(\cdot,\psi_1)$  and $ \varphi_2 \in\text{supp} {\mathcal L}(\cdot,\psi_2)$ such that
 there exists $t_0\in[0,1]$ such that
$$
t_0\varphi_1(x)+(1-t_0)\varphi_2(x)\geq \alpha \ \ \ \ \ \forall \ \ x\in X.
$$
\item[(ii)]	$$\inf\limits_{x\in X} p(x,y_0)=\inf_{x\in X}\sup_{\psi\in\Psi}{\mathcal L}(x, \psi)=\sup_{\psi\in\Psi}\inf_{x\in X}{\mathcal L}(x, \psi).$$
Moreover if there exists $\bar{\psi}\in \Psi$ such that $V\geq \bar{\psi}$ the conditions $(i)$ and $(ii)$ are equivalent to 
\item[(iii)] $V$ is lower semicontinuous at $y_0$.
\end{description}
	\end{theorem}
\begin{proof}
$(i) \rightarrow (ii)$ By Proposition \ref{propoptvalue} $\inf\limits_{x\in X}\sup\limits_{\psi\in\Psi}{\mathcal L}(x, \psi)=V(y_0)$.  Take any $\varepsilon>0$ and $\alpha=V(y_0)-\varepsilon/2$. By $(i)$ there exist $\psi_1,\psi_2\in \Psi$ and  $ \varphi_1 \in\text{supp} {\mathcal L}(\cdot,\psi_1)$  and $ \varphi_2 \in\text{supp} {\mathcal L}(\cdot,\psi_2)$ such that
 there exists $t_0\in[0,1]$ such that
$$
t_0\varphi_1(x)+(1-t_0)\varphi_2(x)\geq \alpha \ \ \ \ \ \forall \ \ x\in X.
$$
Since ${\mathcal L}(x,\psi)=\psi(y_0)-p^*_x(\psi)$ we have
$$
\psi_1(y_0)-p^*_x(\psi_1)=\psi_1(y_0)+\inf_{y\in Y}\{p(x,y)-\psi_1(y)\}\geq \varphi_1(x) \ \ \ \forall \ x\in X
$$
$$
\psi_2(y_0)-p^*_x(\psi_2)=\psi_2(y_0)+\inf_{y\in Y}\{p(x,y)-\psi_2(y)\}\geq \varphi_2(x) \ \ \ \forall \ x\in X
$$
If $t_0\neq 0$ then 
$$
t_0\psi_1(y_0)+t_0p(x,y)-t_0\psi_1(y)\geq t_0\varphi_1(x) \ \ \ \forall \ x\in X, y\in Y 
$$
$$
(1-t_0)\psi_2(y_0)+(1-t_0)p(x,y)-(1-t_0)\psi_2(y)\geq (1-t_0)\varphi_2(x) \ \ \ \forall \ x\in X, y\in Y. 
$$
Hence
$$
p(x,y)+ t_0(\psi_1(y_0)-\psi_1(y))+ (1-t_0)(\psi_2(y_0)-\psi_2(y))\geq \alpha \ \ \ \forall \ x\in X, y\in Y. 
$$
\begin{equation} 
\label{perturbation}
p(x,y)\geq  t_0(\psi_1(y)-\psi_1(y_0))+ (1-t_0)(\psi_2(y)-\psi_2(y_0))+ \alpha \ \ \ \forall \ x\in X, y\in Y. 
\end{equation}
Since $\psi_{1},\psi_{2}\in\Psi$ are lower semicontinuous at $y_{0}$ there exist neighbourhoods $W_{1}(y_{0})$ and $W_{2}(y_{0})$ such that
$$
\psi_{1}(y)-\psi_{1}(y_{0})>-\varepsilon/2 \text{  for  }y\in W_{1}(y_{0})
$$
$$
\psi_{2}(y)-\psi_{2}(y_{0})>-\varepsilon/2 \text{  for  }y\in W_{2}(y_{0})
$$
Hence, by \eqref{perturbation},
$$
p(x,y)\geq   V(y_{0})-\varepsilon/2 -\varepsilon/2\ \ \ \forall \ x\in X, y\in W_{1}(y_{0})\cap W_{2}(y_{0})\
$$
and finally
$$
V(y)>V(y_{0})-\varepsilon \text{   for  }  y \in W_{1}(y_{0})\cap W_{2}(y_{0}).
$$
If $t_{0}=0$, then 
$$
\psi_2(y_0)+p(x,y)-\psi_2(y)\geq \alpha \ \ \ \forall \ x\in X, y\in Y. 
$$
equivalently
$$
p(x,y)\geq  \psi_2(y) -\psi_2(y_0)+\alpha \ \ \ \forall \ x\in X, y\in Y. 
$$
Since $\psi_{2}\in\Psi$ is lower semicontinuous at $y_{0}$ there exist neighbourhood $W_{2}(y_{0})$ such that
$$
\psi_{2}(y)-\psi_{2}(y_{0})>-\varepsilon/2 \text{  for  }y\in W_{2}(y_{0})
$$
Hence,
$$
p(x,y)\geq   V(y_{0})-\varepsilon/2 -\varepsilon/2\ \ \ \forall \ x\in X, y\in W_{2}(y_{0})\
$$
and finally
$$
V(y)>V(y_{0})-\varepsilon \text{   for  }  y \in W_{2}(y_{0}).
$$

$(ii)\rightarrow (iii)$ By Proposition \ref{rol-pro} we get that $V$ is $\Psi$-convex at $y_0$, which means that $V(y_0)=V^{**}(y_0)$, by Proposition \ref{propoptvalue} we get desired conclusion.

$(iii)\rightarrow (i)$   Let  $\alpha <\inf\limits_{x\in X}\sup\limits_{\psi\in \Psi}{\mathcal L}(x, \psi)$.  By equality,
	$
	\inf\limits_{x\in X}\sup\limits_{\psi\in\Psi}{\mathcal L}(x, \psi)=\sup\limits_{\psi\in\Psi}\inf\limits_{x\in X}{\mathcal L}(x, \psi),$
	we get
	$$
	\sup_{\psi\in \Psi}\inf_{x\in X}{\mathcal L}(x, \psi)>\alpha.
	$$
	So, there exists $\bar{\psi}\in \Psi$ such that
	$$
	{\mathcal L}(x, \bar{\psi})\geq \alpha \ \ \text{for all} \ \ \ x\in X.
	$$
	Thus, the function
	$\bar{\varphi}:= \alpha$ belongs to the support set $\text{supp}\, {\mathcal L}(\cdot, \bar{\psi})$.
	By the fact that  $[\bar{\varphi}<\alpha]=\emptyset$,
	we get that, for all $\varphi \in \Phi$, the functions $\bar{\varphi}$ and $\varphi$ have the intersection property on $X$ at the level $\alpha$. From Lemma \ref{key_lemma} we get the condition  $(i)$.
\end{proof}
\begin{remark} Since the class $\Psi$ of elementary functions consist of the real-valued functions the assumption that there exists $\bar{\psi}\in \Psi$ such that $V\geq \bar{\psi}$ means that $V(y)>-\infty$ for every $y\in Y$.
\end{remark}

\begin{definition}(Definition 6.3, \cite{rubbook})
We say that the family ${\cal G}$ has a {\em support to a Urysohn peak} at $y_0$ if the following conditions hold
\begin{itemize}
\item[(i)] ${\cal G}$ is a conic set of continuous functions defined on a metric space $Y$ such that for all $g\in {\cal G}$ and $c\in\mathbb{R}$
$g+c\in \mathbb{R}$.
\item[(ii)] for each $\varepsilon>0$ and $\delta>0$ there exists a function $g\in {\cal G}$ such that
\begin{equation}
g(y_0)>1-\varepsilon,\ \ \ 
g(y)\leq 1 \ \ \text{if} \ \ d(y,y_0)<\delta, \ \ \ g(y)\leq 0 \ \ \text{if} \ \ d(y,y_0)\geq \delta.
\end{equation}
\end{itemize}
\end{definition}
The Urysohn peak property is equivalent to the peaking property for positively homogeneous  classes of elementary functions. 
\begin{proposition}{(\cite{rubbook}, Lemma 6.1)}
\label{peak}
Assume that $\Psi$ has support to a Urysohn peak at $y_0$ and assume that function $h:Y\rightarrow \mathbb{R}$ is lower semicontinuous at $y_0$ and there exists $\psi\in \Psi$ such that $\psi< h$. Then $h$ is $\Psi$-convex.    
\end{proposition}

In view of Theorem \ref{theorempeak} it is important to find classes $\Psi$ which possess the peaking property. Below we present examples of such classes
\begin{example}
\label{ex_1}
\begin{enumerate}
\item The class $\Psi_{lsc}$, defined by \eqref{philsc},
where $Y$ is a Hilbert, has the Urysohn peak property at every $y\in Y$ (e.g. \cite{rubbook}, Example 6.2)
\item Let $Y$ be a Banach space which contains a closed convex bounded set $C$ with $0\in \text{int}C$  such that each boundary point of $C$ is also a strictly exposed point of $C$. Let $\mu_C$ be the Minkowski gauge of the set $C$, i.e $\mu_C(y):=\inf_{\lambda\geq 0}\{y\in \lambda C\}$. Then the class
$$
\Psi_{M}:=\{ \psi:Y\rightarrow\mathbb{R}\ : \ \psi(y)=-a\mu_C(y)+l(y), \ a> 0, l\in Y^* \}
$$
has the support to Urysohn peak property at every $y\in Y$. (see Theorem 6.3 \cite{rubbook})
\item Let $(X,d)$ be a metric space. Let $g:[0,+\infty)\rightarrow [0,+\infty)$ be a function such that $g(0)=0$ and there exists a constant $C>0$ such that for $0\leq t,s \leq \delta$ the  following condition holds
\begin{equation}
    g(t+s)\leq C(g(t)+g(s))
\end{equation}
The class 
\begin{equation}
\Psi_{g,d}:= \{\varphi : Y \rightarrow \mathbb{R}, \ \varphi(y)=-a g(d(y,y_{0}))+c, \ \ y_{0}\in Y,\  \ a>, \ c\in \mathbb{R} \},
\end{equation}
has a peaking property at every $y\in Y$ (\cite{rolewicz}, Proposition 2.1.4). Let us note that for $g(x)=x$ the class $\Psi_{g,d}$ is equivalent to the class defined by \eqref{phid}.
\end{enumerate}
\end{example}

\section{Examples}
In the present section, we specify our zero duality results for two classes of elementary functions $\Phi$ (multipliers), i.e. $\Phi_{d}$ and $\Phi_{lsc}$. Finally, in subsection \ref{sub_linear}, we consider the duality for linear infinite-dimensional programming problems; in particular, we provide a short proof of Kantorovich duality.
 
 Let $X$ be a nonempty set and $(Y,d)$ be a metric space.
 Consider  the constrained optimization problem of the form
	\begin{equation}
	\label{problem211}
	\text{Minimize}_{x\in X}\ \ \ \ f(x) \ \ \ \ \ \ x\in A(y_{0}),
	\end{equation}
	where $f:X\rightarrow(-\infty,+\infty]$ is a proper function, {\em i.e.,}  $\text{dom\,} f\neq \emptyset$ and $A: Y\rightrightarrows X$ is a set-valued mapping with $\text{dom\, } A:=\{y\in Y\ :\ A(y)\neq\emptyset\}=Y.$
	The  family of parametrized/perturbed problems $(P_y)$
	\begin{equation}
	\label{problem221}
	\text{Minimize}\ \ \ \ p(x,y) \ \ \ \ \ \ x\in X
	\end{equation}
	is based on the perturbation/parametrization function $p:X\times Y\rightarrow (-\infty,+\infty]$  defined as
 (see \cite{rocka})  
		\begin{equation} 
		\label{funkcjap}
		p(x,y)=\left \{ 
	\begin{matrix}
	f(x) , \ \ \ \ \ \ \  \ x\in A(y)\cr
	+\infty, \ \ \  \ \ \ \ \ x\notin A(y)
	\end{matrix}
	\right.,
	\end{equation}
{\em i.e.,}  $p(x,y)=f(x)$, whenever $x\in A(y)$, for any $x\in X$ and $y\in Y$.

\subsection{Class $\Psi_{d}$} 

 Let $X$ be a set and let $(Y,d)$ be a metric space with the metric $d$. Let $\Psi_d$ be the class of functions $\psi:Y\rightarrow\mathbb{R}$ defined by \eqref{example5}. The Lagrangian \eqref{genlag} , ${\mathcal L}:X\times \Psi_{d}\rightarrow [-\infty, +\infty]$ takes the form
$$
{\mathcal L}(x, \psi)=-ad(y_0,\bar{y})-\sup_{y\in Y}\{ -ad(y,\bar{y})-p(x,y)\},
$$
then
$$
{\mathcal L}(x, \psi)={\mathcal L}(x, \bar{y},a)=-ad(y_0,\bar{y})-\sup_{y\in G(x)}\{ -ad(y,\bar{y})-f(x)\}=-ad(y_0,\bar{y})+f(x)+\inf_{y\in G(x)}\{ad(y,\bar{y})\}
$$
where the set-valued mapping $G:X\rightrightarrows Y$ is the inverse to $A$, {\em i.e.,} $G:=A^{-1}$. 

The term
\begin{equation} 
\label{metric_projection}
\inf_{y\in G(x)}\ d(y,\bar{y}).
\end{equation}
represents the metric projection of $\bar{y}$ onto $G(x)$. 
Assume that $G(x)$ is closed for any $x\in X$. 
Hence, for any $x\in X$, the Lagrangian primal objective function takes the following forms.

 The elements of class $\Psi_{d}$ can be identified with elements of the set $\mathbb{R}_{+}\times Y$. In view of this,
\begin{equation} 
\label{eq:lagr_primal}
\begin{array}{l}
\sup_{(\bar{y},a)\in Y\times\mathbb{R}_{+}}\mathcal{L}(x,\bar{y},a)\\
=\sup_{(\bar{y},a)\in Y\times\mathbb{R}_{+}}-ad(y_0,\bar{y})+f(x)+\inf_{y\in G(x)} \ ad(y,\bar{y})\\
=f(x)+\sup_{(\bar{y},a)\in Y\times\mathbb{R}_{+}}\{-ad(y_0,\bar{y})+\inf_{y\in G(x)}\ ad(y,\bar{y})\}\\
= f(x)+\left\{\begin{array}{ll}
0&y_{0}\in G(x)\Leftrightarrow\ x\in A(y_{0})\\
+\infty&y_{0}\notin G(x)\Leftrightarrow\ x\notin A(y_{0})\\
\end{array}\right.
\end{array}
\end{equation}
To see the latter equality, take any fixed $x\in X$. 
\begin{enumerate} 
\item  $y_{0}\in G(x)$. By \eqref{metric_projection}, 
for any $\bar{y}\in Y$ it holds $\inf_{y\in G(x)} d(y,\bar{y})\le d(y_{0},\bar{y})$
and
$$
\sup_{(\bar{y},a)\in Y\times\mathbb{R}_{+}}\{-ad(y_0,\bar{y})+\inf_{y\in G(x)}\ ad(y,\bar{y})\}\le 0.
$$
Moreover, for $(\bar{y},a)=(y_{0},a)$ we have $ -ad(y_0,\bar{y})+\inf_{y\in G(x)}\ ad(y,\bar{y})=0$, i.e.
$$
\sup_{(\bar{y},a)\in Y\times\mathbb{R}_{+}}\{-ad(y_0,\bar{y})+\inf_{y\in G(x)}\ ad(y,\bar{y})\}= 0.
$$
\item  $y_{0}\notin G(x)$. Since $G(x)$ is a closed set in $Y$, then
$\inf_{y\in G(x)}\ d(y,y_{0})>0$. For 
$(y_{0},a_{n})\in Y\times\mathbb{R}_{+}$, with $a_{n}\rightarrow+\infty$ we have
$$
-a_{n}d(y_0,y_{0})+\inf_{y\in G(x)}\ a_{n}d(y,y_{0})\rightarrow+\infty.
$$
\end{enumerate}
This proves that
\begin{equation} 
\label{primalagrangian}
 f(x)=\sup_{(\bar{y},a)\in Y\times\mathbb{R}_{+}}\mathcal{L}(x,\bar{y},a)\ \ \forall\ x\in X,
\end{equation}
{\em i.e.,} the original problem \ref{problem211} is equivalent to the Lagrangian primal.

By Example 5.7.3 the class $\Psi_{d}$ has a peaking property. 

\begin{theorem}
\label{theorempeakex}
		Let $X$ be a nonempty set, and $(Y,d)$ be a metric space. Let $\Phi$ be a class of elementary functions $\varphi:X\rightarrow \mathbb{R}$. Assume that ${\mathcal L}(\cdot, \psi)$ is $\Phi$-convex on $X$ for every $\psi\in \Psi_{d}$. Then 
	$$\inf\limits_{x\in X} f(x)=\inf_{x\in X}\sup_{\psi\in\Psi}{\mathcal L}(x, \psi)=\sup_{\psi\in\Psi}\inf_{x\in X}{\mathcal L}(x, \psi).$$
	\end{theorem}
\begin{proof}
By Theorem \ref{theorempeak} it is enough to show the intersection property.

Let $a>0$ and $\psi_0(y):=-ad(y,y_0)$. For all $x\in X$ we have
$$
{\mathcal L}(x, \psi_0)={\mathcal L}(x, y_0,a)= -ad(y_0,y_0)+f(x)+\inf_{y\in G(x)}\{ad(y,y_0)\}=f(x)+\inf_{y\in G(x)}\{ad(y,y_0)\}\geq f(x).
$$
$$
\geq \inf_{x\in X}f(x) = \inf_{x\in X} \sup_{\psi\in \Psi} \mathcal{L}(x,\psi)
$$
Then
$$
{\mathcal L}(x, \psi_0)\geq \inf_{x\in X}f(x)= \inf_{x\in X}\sup_{\psi\in\Psi}{\mathcal L}(x, \psi),
$$
where the equality comes from Proposition \eqref{propoptvalue}. Take any $\alpha < \inf\limits_{x\in X}\sup\limits_{\psi\in\Psi}{\mathcal L}(x, \psi)$. Then 
$$
{\mathcal L}(x, \psi_0)\geq \alpha. 
$$
Let $\varphi_0(x)\equiv \alpha$, $\varphi_0$ has the intersection property with any other $\varphi$.  The conclusion follows from Theorem \ref{theorempeak}.
\end{proof}

\subsection{Class $\Psi_{lsc}$ } 

Let $X$ be a set and let $Y$ be a Hilbert space. The elements of class $\Psi_{lsc}$ defined by \eqref{example3} can be identified with triples $(a,u,c)\in\mathbb{R}_{+}\times Y\times\mathbb{R}$.

 In this case, the Lagrangian \eqref{genlag} , ${\mathcal L}:X\times \Psi_{lsc}\rightarrow [-\infty, +\infty]$ takes the form
$$
{\mathcal L}(x, \psi)=-a\|y_{0}\|^{2}+\langle u,y_{0}\rangle+c-\sup_{y\in Y}\{ -a\|y\|^{2}+\langle u,y\rangle+c-p(x,y)\},
$$
then
$$
{\mathcal L}(x, \psi)={\mathcal L}(x, u,a)=-a\|y_0\|^{2}+\langle u,y_{0}\rangle-\sup_{y\in G(x)}\{ -a\|y\|^{2}+\langle u,y\rangle-f(x)\}
$$
where the set-valued mapping $G:X\rightrightarrows Y$ is the inverse to $A$, {\em i.e.,} $G:=A^{-1}$. 

The class $\Psi_{lsc}$ is a convex cone, hence Theorem 4.1 can be used.

Assume now that, for each $x\in X$, the set $G(x)\subset Y$ is $\Phi_{lsc}$-convex in the sense of Pallaschke \& Rolewicz( book page 34), {\em i.e.},
a set $C\subset Y$ is said to be {\em $\Phi_{lsc}$-convex}  if for every $p\not\in C$ there is $\varphi_{p}\in\Phi_{lsc}$  such that
\begin{equation} 
\label{eq:(1.4.4)}
\varphi_{p}(p)>0,
\end{equation}
and
\begin{equation} 
\label{eq:(1.4.5)}
\varphi_{p}(y)\le 0\ \ \ \text{for  } y\in C.
\end{equation}

We have
\begin{equation} 
\label{eq:lagr_primal1}
\begin{array}{l}
\sup_{(a,u,c)\in \mathbb{R}_{+}\times Y\times\mathbb{R}}\mathcal{L}(x,a,u,c)\\
=\sup_{(a,u)\in \mathbb{R}_{+}\times Y}-a\|y_{0}\|^{2}+\langle u\ |\ y_{0}\rangle +f(x)
-\sup_{y\in G(x)} \ -a\|y\|^{2}-\langle u, y\rangle\\
=f(x)+\sup_{(a,u)\in \mathbb{R}_{+}\times Y}\{-a\|y_0\|^{2}+\langle u, y_{0}\rangle -\sup_{y\in G(x)}\ -a\|y\|^{2}+\langle u, y\rangle\}\\
= f(x)+\left\{\begin{array}{ll}
0&y_{0}\in G(x)\Leftrightarrow\ x\in A(y_{0})\\
+\infty&y_{0}\notin G(x)\Leftrightarrow\ x\notin A(y_{0})\\
\end{array}\right.
\end{array}
\end{equation}
Indeed, to see the latter equality, we start by taking $y_{0}\in G(x)$. Then, for any $(a,u)\in\mathbb{R}_{+}\times Y$ we have
$$
-a\|y_0\|^{2}+\langle u, y_{0}\rangle -\sup_{y\in G(x)}\ -a\|y\|^{2}+\langle u, y\rangle\le 0,
$$
which proves the first part of the last equality. Assume now that $y_{0}\not\in G(x)$. By \eqref{eq:(1.4.4)}-\eqref{eq:(1.4.5)}, there exists $(a_{0},u_{0})\in\mathbb{R}_{+}\times Y$ such that
$$
-a_{0}\|y_0\|^{2}+\langle u_{0}, y_{0}\rangle -\sup_{y\in G(x)}\ -a_{0}\|y\|^{2}+\langle u_{0}, y\rangle> 0.
$$
For any $\lambda_{n}\rightarrow+\infty$ we get $(\lambda_{n}a _{0}, \lambda _{n}u_{0})\in\mathbb{R}_{+}\times Y$ and
$$
\lambda_{n}[-a\|y_0\|^{2}+\langle u, y_{0}\rangle -\sup_{y\in G(x)}\ -a\|y\|^{2}+\langle u, y\rangle ]\rightarrow+\infty\ \ \text{when  } n\rightarrow+\infty,
$$
which proves the second part of the last equality in \eqref{eq:lagr_primal1}. This proves that for the class $\Psi_{lsc}$ the primal problem and Lagrangian primal problems coincide provided $\Phi\equiv \Phi_{lsc}$.

\subsection{Infinite-dimensional conic linear programming}
\label{sub_linear}

Let $F$ be Hausdorff locally convex space, $F^*$ is its topological dual space and  $\pi\in F^*$. Consider the following problem
\begin{equation}
	\label{conic}
	\text{Minimize}\ \ \ \ \langle  \pi,f \rangle \ \ \ \ \ \ f\in F, \ \ f-c\in Q,
	\end{equation}
where, $Q\subseteq F$ is convex cone and $c\in Q$. 

Let $ind_{Q}:F\rightarrow (-\infty,+\infty]$ be  indicator functions of cone  $Q$. We consider  the perturbation function $p:F\times F\rightarrow (-\infty,+\infty] $, defined as follows
$$
p(f,q):=\langle  \pi,f \rangle+ind_{Q}(f-c-q).
$$
Let $q_0=0$, then the problem \eqref{conic} takes the form
\begin{equation}
	\label{conic1}
	\text{Minimize}_{f\in F}\ \ \ \ p(f,0).
	\end{equation}
Let $\Psi=F^*$. The Lagrangian defined by \eqref{genlag} takes the form
$$
{\mathcal L}(f, q^*)=-p^*_f(q^*),
$$
we have
$$
{\mathcal L}(x, y^*)=-\sup_{q\in F}\{ \langle q^*,f \rangle - \langle  \pi,f \rangle-ind_{Q}(f-c-q) \}=
$$
$$
=\langle  \pi,f \rangle -\langle q^*,f \rangle +\inf_{q\in F} \{ ind_{Q}(f-c-q)\},
$$
The Lagrangian dual problem takes the form
$$
\sup_{q^*\in F^*}\inf_{f\in F}{\mathcal L}(f, q^*)=
\sup_{q^*\in F^*}\inf_{f\in F}\{  \langle  \pi,f \rangle -\langle q^*,f \rangle +\inf_{q\in F} \{ ind_{Q}(f-c-q)\}\}
$$
$$
=\sup_{q^*\in F^*}\inf_{f\in F}\{  -\langle f,0 \rangle+ \langle  \pi,f \rangle  -\langle q^*,f \rangle +\inf_{q\in F} \{ ind_{Q}(f-c-q)\}\}
$$
$$
=\sup_{q^*\in F^*}\inf_{f\in F}\inf_{q\in F}\{ - \langle f,0 \rangle+ \langle  \pi,f \rangle -\langle q^*,f\rangle + ind_{Q}(f-c-q)\}
$$
$$
=\sup_{q^*\in F^*}(-\sup_{f\in F}\sup_{q\in F}\{  \langle f,0 \rangle +\langle q^*,f\rangle- \langle  \pi,f \rangle  - ind_{Q}(f-c-q)\}
$$
\begin{equation}
\label{dualzal}
=\sup_{q^*\in F^*}\{-p^*(0,q^*)\}
\end{equation}
Note that the \eqref{dualzal} coincides with the problem $(P^*_{c,b})$  from \cite{zalinescu}.
\\
\subsubsection{Kantorovich duality}
Let ${\cal X}$ be a locally compact topological space and $C_{0}({\cal X})$ spaces of continuous functions vanishing at infinity. Recall that $u:{\cal X}\rightarrow\mathbb{R}$ is said to vanish at infinity if, for every $\varepsilon>0$, there exists a compact set $K\subseteq X$ such that
$$
|u(x)|\leq \varepsilon, \ \ \ \ \forall \ x\in {\cal X}\setminus K.
$$
The spaces $C_{0}({\cal X})$ is equipped with the supremum norm
$$
\|u\|_{\infty}=\sup_{x\in {\cal X}}|u(x)|.
$$
We have 
\begin{equation} 
\label{dual_space}
(C_{0}({\cal X}))^*={\cal M}_{b}({\cal X}),
\end{equation}
where ${\cal M}_{b}({\cal X})$ is the space of signed Radon measures $\mu$ on ${\cal X}$ such that $|\mu|$ is finite with the total variation norm defined by
$$
\|\mu\|_{{\cal M}_{b}({\cal X})}=\int_{X}\ d|\mu|,
$$
see e.g., Theorem 1.88, \cite{gasinski}.

Let  ${\cal Y}$ be a locally compact topological space. When
$$
F:=C_{0}({\cal X})\oplus C_{0}({\cal Y}).
$$
		Recall that
	$$
	C_{0}({\cal X})\oplus C_{0}({\cal Y})=\{f\in C_{0}({\cal X}\times {\cal Y})\ |\ f(x,y)=\psi(x)+\phi(y),\ \psi\in C_{0}({\cal X}),\ \phi\in C_{0}({\cal Y})\}
	$$
 and 
	\begin{equation} 
	\label{eq:dual_direct}
	(C_{0}({\cal X})\oplus C_{0}({\cal Y}))^{*}={\cal M}_{b}({\cal X})\oplus {\cal M}_{b}({\cal Y})
	\end{equation}
  By \eqref{eq:dual_direct}, any linear functional of $C_{0}({\cal X})\oplus C_{0}({\cal Y})$,
 is defined by a pair of signed Radon measures $\mu\in{\cal M}_{b}({\cal X})$, $\nu\in{\cal M}_{b}({\cal Y})$ such that for each $f\in C_{0}({\cal X})\oplus C_{0}({\cal Y})$, $f=(\psi,\phi)$,
 $$
 \langle (\mu,\nu), f\rangle:=\int_{X}\psi(y)d\mu+\int_{Y}\phi(y)d\nu
 $$
 Now, fix $\mu\in{\cal M}_{b}({\cal X})$, $\nu\in{\cal M}_{b}({\cal Y})$ and $c\in C_{0}({\cal X}\times {\cal Y})$. Consider the problem
 \begin{equation} 
 \label{eq:linear_direct}
 \tag{LPD}
 \begin{array}{l}
 \text{Minimize}_{(\psi,\phi)\in C_{0}({\cal X})\oplus C_{0}({\cal Y})} -\langle (\mu,\nu),(\psi,\phi)\rangle\\
 subject \ to\ \psi(x)+\phi(y)\le c(x,y)\ \forall\ (x,y)\in {\cal X}\times {\cal Y}
 \end{array}
 \end{equation}
 Equivalently,
 \begin{equation} 
 \label{eq:linear_direct1}
 \tag{LPD1}
 \begin{array}{l}
 \text{Minimize}_{(\psi,\phi)\in C_{0}({\cal X})\oplus C_{0}({\cal Y})} -[\int_{X}\psi(x)\ d\mu+\int_{Y}\phi(y)\ d\nu]\\
 subject \ to\ \psi(x)+\phi(y)\le c(x,y)\ \forall\ (x,y)\in {\cal X}\times {\cal Y},
 \end{array}
 \end{equation}
 or,
 \begin{equation} 
 \label{eq:linear_direct2}
 \tag{LPD2}
 \begin{array}{l}
 \text{Minimize}_{(\psi,\phi)\in C_{0}({\cal X})\oplus C_{0}({\cal Y})} -[\int_{X}\psi(x)\ d\mu+\int_{Y}\phi(y)\ d\nu]\\
 subject \ to\ \psi+\phi- c\in Q,
 \end{array}
 \end{equation}
 where $Q:=C_{0}^{-}({\cal X}\times {\cal Y})=\{f\in C_{0}^{-}({\cal X}\times {\cal Y})\ |\ f\le 0\} $. This is a conic linear optimization problem of the from \eqref{conic}, i.e.
 \begin{equation} 
 \label{eq:linear_direct_conic}
 \begin{array}{l}
 \text{Minimize}_{(\psi,\phi)\in C_{0}({\cal X})\oplus C_{0}({\cal Y})} -[\langle\mu,\psi\rangle+\langle\nu,\phi\rangle]\\
 subject \ to\ \psi+\phi- c\in Q,
 \end{array}
 \end{equation}
 
 The perturbation function $p:[C_{0}({\cal X})\oplus C_{0}({\cal Y})]\times C_{0}({\cal X}\times {\cal Y})\rightarrow(-\infty,+\infty]$
 \begin{equation} 
 \label{eq:pert_direct} 
 p((\psi,\phi),q):=-\langle (\mu,\nu), (\psi,\phi)\rangle+ind_{Q}((\psi,\phi)-c+q).
 \end{equation}
 
 By \eqref{eq_partial_conjugate}, the  $p_{(\psi,\phi)}^{*}$ is the (partial) conjugate of $p$ with respect to the second variable, $p_{(\psi,\phi)}^{*}:{\cal M}_{b}({\cal X}\times {\cal Y})\rightarrow (-\infty,+\infty)$,
 for any $q^{*}\in {\cal M}_{b}({\cal X}\times {\cal Y})$,
$$
 p_{(\psi,\phi)}^{*}(q^{*})\begin{array}[t]{l}
 =\sup\limits_{q\in C_{0}({\cal X}\times {\cal Y})}\{\langle q^{*},q\rangle +\langle (\mu,\nu),(\psi,\phi)\rangle-ind_{Q}((\psi,\phi)-c+q)\}\\
= \langle (\mu,\nu),(\psi,\phi)\rangle+\sup_{q\in C_{0}({\cal X}\times {\cal Y})}\{\langle q^{*},q\rangle -ind_{Q}((\psi,\phi)-c+q)\}.\\
 \end{array}
 $$
 For any $q \in C_{0}({\cal X}\times {\cal Y})$, by the definition of indicator function $ind$,
 $$
 \langle q^{*},q\rangle -ind_{Q}((\psi,\phi)-c+q)=\left\{\begin{array}{ll}
 \langle q^{*},q\rangle&(\psi,\phi)-c+q\in Q\\
 -\infty&otherwise
 \end{array}\right.
 $$
 
 $$
 \sup_{q\in C_{0}(X\times Y)}\{\langle q^{*},q\rangle -ind_{Q}((\psi,\phi)-c+q)\}=\sup_{q\in C_{0}({\cal X}\times {\cal Y}),(\psi,\phi)-c+q\in Q}\langle q^{*},q\rangle
 $$

 Now
 $$
 (\psi,\phi)-c+q\in Q\ \leftrightarrow\ \ (\psi,\phi)-c+q=k\in Q
 $$
 and
 $$
\langle q^*,q \rangle= \langle q^{*},k+c-(\psi,\phi)\rangle=\langle q^{*},k\rangle+\langle q^{*},c-(\psi,\phi)\rangle.
 $$
 Hence,
 $$
 \sup_{q\in C_{0}({\cal X}\times {\cal Y}),(\psi,\phi)-c+q\in Q}\langle q^{*},q\rangle
 =\left\{\begin{array}{ll}
 \langle q^{*}, c-(\psi,\phi)\rangle&q^{*}\in{\cal M}_{b}^{+}({\cal X}\times {\cal Y})\\
 +\infty&otherwise\  (\langle q^{*},k\rangle>0 \ for\ some\ k\in Q)
 \end{array}
 \right.,
 $$
 where 
 $${\cal M}_{b}^{+}({\cal X}\times {\cal X})=\{ \mu \in {\cal M}_{b}({\cal X}\times {\cal Y})\ | \ \int_{{\cal X}\times {\cal Y}} k d\mu\geq 0, \ \forall \ k\in Q^{+}\}=\{ \mu \in {\cal M}_{b}({\cal X}\times {\cal Y})\ | \ \mu\ge 0\}
 $$
 is the classical nonnegative dual cone to $Q^{+}:=C_{0}({\cal X}\times {\cal Y})^{+}=\{f\in C_{0}({\cal X}\times {\cal Y})\ |\ f\ge 0\}$. 
 Hence,
 $$
 p^{*}_{(\psi,\phi)}(q^{*})=\left\{\begin{array}{ll}
 \langle (\mu,\nu), (\psi,\phi)\rangle+\langle q^{*}, c-(\psi,\phi)\rangle&q^{*}\in{\cal M}_{b}^{+}({\cal X}\times {\cal Y})\\
 +\infty&otherwise\ ( \langle q^{*},k\rangle>0 \ for\ some\ k\in Q)
 \end{array}
 \right.
 $$
According to \eqref{genlag}, the Lagrangean ${\cal L}((\psi,\phi),q^{*}):C_{0}({\cal X}\times {\cal Y})\times{\cal M}_{b}({\cal X}\times {\cal Y})\rightarrow[-\infty,+\infty]$ takes the form
\begin{equation} 
 \label{eq:lagr_direct_1}
 \begin{array}{l}
 {\cal L}((\psi,\phi),q^{*})
 =-p^{*}_{(\psi,\phi)}(q^{*})\\
 =\left\{\begin{array}{ll}
 -\langle (\mu,\nu),(\psi,\phi)\rangle+\langle q^{*}, (\psi,\phi)-c\rangle&q^{*}\in{\cal M}_{b}^{+}({\cal X}\times {\cal Y})\\
 -\infty&otherwise\ ( \langle q^{*},k\rangle>0 \ for\ some\ k\in Q)
 \end{array}
 \right.
 \end{array}
 \end{equation}
 
 Equivalently, 
 \begin{equation} 
 \label{eq:lagr_direct}
 \begin{array}{l}
 {\cal L}((\psi,\phi),q^{*})\\
 =\left\{\begin{array}{ll}
 \langle q^{*},-c\rangle, & q^{*}\in{\cal M}_{b}^{+}({\cal X}\times {\cal Y}) \land  q^{*}\in\Pi(\mu,\nu)\\
-\langle (\mu,\nu),(\psi,\phi)\rangle+\langle q^{*}, (\psi,\phi)-c\rangle&q^{*}\in{\cal M}_{b}^{+}({\cal X}\times {\cal Y}) \land  q^{*}\notin\Pi(\mu,\nu)\\
 -\infty&otherwise\ ( \langle q^{*},k\rangle>0 \ for\ some\ k\in Q)
 \end{array}
 \right.
 \end{array}
 \end{equation}
 
According to \eqref{lprob},  Lagrangean primal related to the Lagrangean \eqref{eq:lagr_direct} is as follows 
  \begin{equation} 
  \label{eq:lag_prim_direct}
  \inf_{(\psi,\phi)\in C_{0}({\cal X}\times {\cal Y})}\sup_{q^{*}\in{\cal M}_{b}({\cal X}\times {\cal Y})}{\cal L}((\psi,\phi),q^{*})
  \end{equation}
Observe,
\begin{equation}
\label{lag1}
\sup_{q^{*}\in{\cal M}_{b}({\cal X}\times {\cal Y})}{\cal L}((\varphi,\psi),q^{*})=\sup_{q^{*}\in{\cal M}_{b}^{+}({\cal X}\times {\cal Y})}{\cal L}((\varphi,\psi),q^{*})=\left\{\begin{array}{ll}
-\langle (\mu,\nu),(\psi,\phi)\rangle&(\psi,\phi)-c\in Q\\
+\infty&(\psi,\phi)-c\not\in Q
\end{array}\right.
\end{equation}
Finally, the Lagrangean primal takes the form
\begin{equation} 
  \label{lag2}
  \tag{LP}
   \inf_{(\psi,\phi)-c\in Q}\sup_{q^{*}\in{\cal M}_{b}({\cal X}\times {\cal Y})}{\cal L}((\psi,\phi),q^{*})
  \end{equation}


To formulate  the Lagrangean dual, 
 according to \eqref{ldual}, the Lagrangean dual related to the Lagrangean \eqref{eq:lagr_direct} is as follows  
   \begin{equation} 
  \label{eq:lag_dual_direct}
  \sup_{q^{*}\in{\cal M}_{b}({\cal X}\times {\cal Y})}\inf_{(\psi,\phi)\in C_{0}({\cal X}\times {\cal Y})}{\cal L}((\psi,\phi),q^{*})
  \end{equation}
 Thus,  for $q^{*}\in{\cal M}_{b}^{+}({\cal X}\times {\cal Y})$ we have
\begin{equation} 
\label{eq:inf_lagrangean1}
\begin{array}{l}
\inf_{(\psi,\phi)\in C_{0}({\cal X}\times {\cal Y})}{\cal L}((\psi,\phi),q^{*})=\\
 =\inf\limits_{(\psi,\phi)\in C_{0}({\cal X}\times {\cal Y})}
-\langle (\mu,\nu),(\psi,\phi)\rangle+\langle q^{*}, (\psi,\phi)-c\rangle\\
 =\inf\limits_{(\psi,\phi)-c\in Q}
-\langle (\mu,\nu),(\psi,\phi)\rangle+\langle q^{*}, (\psi,\phi)-c\rangle,\\
\end{array}
\end{equation}
where the latter equality follows from the fact that 
$$
\langle q^{*}, (\psi,\phi)-c\rangle\le 0\ \ for\ \ (\psi,\phi)-c\in Q
$$
and consequently, 
$$-\langle (\mu,\nu),(\psi,\phi)\rangle+\langle q^{*}, (\psi,\phi)-c\rangle\le -\langle (\mu,\nu),(\psi,\phi)\rangle.
$$ 

   From the above inequality we see that when  $(\psi,\phi)-c\notin Q$, by choosing  $q^{*}\in{\cal M}_{b}({\cal X}\times {\cal Y})$ as a Radon positive measure concentrated at $(\bar{x},\bar{y})$ the dual objective 
   $$
   \inf\limits_{(\psi,\phi)\in C_{0}({\cal X}\times {\cal Y})}{\cal L}((\psi,\phi),q^{*})
   $$
   equals $+\infty$ and consequently, the optimal value of the Lagrangian dual is $+\infty$. 
  
  On the other hand, because the feasible set of the primal problem $\{(\psi,\phi)\ |\ (\psi,\phi)-c\in Q\}\neq\emptyset$, the only possibility for the primal problem \eqref{lag2} to have the optimal value $+\infty$ is that $\sup_{q^{*}\in{\cal M}_{b}^{+}}{\cal L}((\psi,\phi),q^{*})=+\infty$ for any $(\psi,\phi)-c\in Q$ which contradics \eqref{lag1}.
  So, in this case,  there is a duality gap. 
  
  This is why we limit the considerations to  $(\psi,\phi)-c\in Q$. Then the domain of the dual objective function $\inf_{(\psi,\phi)-c\in Q}{\cal L}((\psi,\phi),q^{*})$ is equal to
  $$
  \text{dom}\{\inf_{(\psi,\phi)-c\in Q}{\cal L}((\psi,\phi),q^{*})\}={\cal M}_{b}^{+}({\cal X}\times {\cal Y}).
  $$

  Finally, the Lagrangian dual takes the form
  \begin{equation} 
 \label{eq:lagr_dual_direct3}
 \tag{LD}
 \begin{array}{l}
  \sup_{q^{*}\in{\cal M}_{b}^{+}({\cal X}\times {\cal Y})}\inf_{(\psi,\phi)-c\in Q}{\cal L}((\psi,\phi),q^{*})\\
\end{array}
  \end{equation}
  
\begin{theorem}
\label{th_kantorovich}
Let the Lagrangian function be given by the formula \eqref{eq:lagr_direct}. Then
$$
\sup\limits_{q^{*}\in{\cal M}^{+}_{b}({\cal X}\times {\cal Y})}\inf\limits_{(\psi,\phi)- c\in Q}{\cal L}((\psi,\phi),q^{*})=\inf\limits_{(\psi,\phi)- c\in Q}\sup\limits_{q^{*}\in{\cal M}^{+}_{b}({\cal X}\times {\cal Y})}{\cal L}((\psi,\phi),q^{*})
$$
\end{theorem}
\begin{proof}
We will show that all the assumptions of Theorem \ref{optgen} hold. Let 
$$
\begin{array}{l}
X:=\{ (\psi,\phi)\in F\ |\ (\psi,\phi)- c\in Q\}\\
Y:=C_{0}({\cal X}\times {\cal Y})\\
c\in Y, \ q\in Y,\ \ q^{*}\in {\cal M}_{b}({\cal X}\times {\cal Y})\\
{\cal M}_{b}^{+}({\cal X}\times {\cal Y}) \ \text{dual to } Q\subset C_{0}({\cal X}\times {\cal Y})
\end{array}
$$
and $\Phi=\Psi:= {\cal M}^{+}_{b}({\cal X}\times {\cal Y})$ which is a convex set. The Lagrangian \eqref{eq:lagr_direct} is linear as a function of $(\psi,\phi)$ for every $q^{*}\in{\cal M}_{b}({\cal X}\times {\cal Y})$. 

To complete the proof we need to show the intersection property i.e. that for every $\alpha <\inf\limits_{(\psi,\phi)- c\in Q}\sup\limits_{q^{*}\in{\cal M}^{+}_{b}({\cal X}\times {\cal Y})}{\cal L}((\psi,\phi),q^{*})$ there exist $q^*_1, q^*_2 \in {\cal M}^{+}_{b}({\cal X}\times {\cal Y})$ such that there exists $\ell_1,\in \text{supp} {\cal L}(\cdot,q_1^{*})$ and $\ell_2 \in\text{supp} {\cal L}(\cdot,q_2^{*})$ such that $\ell_1$ and $\ell_2$ have the intersection property at level $\alpha$. 

Take any $\alpha < \inf\limits_{(\psi,\phi)- c\in Q}\sup\limits_{q^{*}\in{\cal M}^{+}_{b}({\cal X}\times {\cal Y})}{\cal L}((\psi,\phi),q^{*})$. 
Let $\bar{q}^*=0$, then
\begin{equation}
\label{inf3}
{\cal L}((\psi,\phi),0)=
 -\langle (\mu,\nu),(\psi,\phi)\rangle \geq \inf\limits_{(\psi,\phi)- c\in Q}\{ -\langle (\mu,\nu),(\psi,\phi)\rangle  \}, \ \ \ \ \ \forall \ \ \ (\psi,\phi)- c\in Q.
\end{equation}
Let us note that for $\psi + \phi-c\in Q$ we have 
$$
\{ -\langle (\mu,\nu),(\psi,\phi)\rangle  \}= \sup_{q^{*}\in{\cal M}_{b}^{+}({\cal X}\times {\cal Y})}{\cal L}((\varphi,\psi),q^{*}).
$$
From this and the inequality \eqref{inf3} we get
$$
{\cal L}((\psi,\phi),0)=
 -\langle (\mu,\nu),(\psi,\phi)\rangle \geq \inf\limits_{\psi + \phi-c\in Q}\sup_{q^{*}\in{\cal M}_{b}^{+}({\cal X}\times {\cal Y})}{\cal L}((\varphi,\psi),q^{*})>\alpha.
$$
Let $\bar{\ell}\equiv\alpha$ the $\ell\in \text{supp} {\cal L}(\cdot,0)$ and $\ell$ have the intersection property with any other function from the set $ \text{supp} {\cal L}(\cdot,q_2^*)$. By Theorem \ref{optgen} we get the required equality.

\end{proof}

\begin{proposition}(section 1.1.1, \cite{villani2003topics})
$q^*\in \Pi(\mu,\nu)$ if and only if $q^*$ is nonegative measure on $X\times Y$ such that for all measurable function $(\psi,\phi)\in C_{b}(X)\times C_{b}(Y)$ we have
$$
\langle (\mu,\nu),(\psi,\phi)\rangle =\langle q^{*}, (\psi,\phi)\rangle\
$$
\end{proposition}
\begin{remark}(Remark 1.26, \cite{villani2003topics})
Let $q^*\in M^{+}_{b}(X\times Y)$. For all measurable function $(\psi,\phi)\in C_{b}(X\times Y)$ we have
$$
\langle (\mu,\nu),(\psi,\phi)\rangle \geq\langle q^{*}, (\psi,\phi)\rangle\
$$
\end{remark}

In Theorem 4 and Theorem 6 we provide sufficient and necessary conditions for zero duality gap for pairs of dual optimization problems involving $\Phi$-convex  and $\Phi_{lsc}$-convex functions. In particular, our results apply to optimization problems where the considered Lagrangian, and the function $p(\cdot,\cdot)$ are paraconvex, or prox-bounded, or DC functions.

Let us observe that Theorem 2,  provides considerable flexibility in choosing Lagrange function ${\mathcal L}$. 
Sufficient and necessary conditions of  Theorem 4 and Theorem 6 are  based on the intersection property (Definition \ref{def_2}), which, in contrast to many existing in the literature conditions, is of purely algebraic character. 


\bibliographystyle{spmpsci}
\bibliography{bibn}

\begin{thebibliography}{10}
\providecommand{\url}[1]{{#1}}
\providecommand{\urlprefix}{URL }
\expandafter\ifx\csname urlstyle\endcsname\relax
  \providecommand{\doi}[1]{DOI~\discretionary{}{}{}#1}\else
  \providecommand{\doi}{DOI~\discretionary{}{}{}\begingroup
  \urlstyle{rm}\Url}\fi

\bibitem{balder}
Balder, E.: An extension of duality-stability relations to nonconvex
  optimization problems.
\newblock SIAM Journal on Control and Optimization \textbf{15} (1977)

\bibitem{bard}
Bardsley, P.: Duality in contracting.
\newblock Discussion paper, University of Melbourne  (2017)

\bibitem{bed-syg}
Bednarczuk, E.M., Syga, M.: Minimax theorems for {$\Phi$}-convex functions with
  applications.
\newblock Control and Cybernetics \textbf{43}(3), 421--437 (2014)

\bibitem{bot}
Bo\c{t}, R.: Conjugate Duality in Convex Optimization, vol. 637 (2010)

\bibitem{Bonnans}
Bonnans, J.F., Shapiro, A.: Perturbation Analysis of Optimization Problems
  (2000)

\bibitem{Bot2003}
Bo{\c{t}}, R.I., Wanka, G.: Duality for composed convex functions with
  applications in location theory, pp. 1--18.
\newblock Deutscher Universit{\"a}tsverlag, Wiesbaden (2003).
\newblock \doi{10.1007/978-3-322-81539-2}

\bibitem{burachik}
Bui, T.H., Burachik, R., Kruger, A., Yost, D.: Zero duality gap in view of
  abstract convexity  (2019)

\bibitem{BurachikRubinov}
Burachik, R.S., Rubinov, A.: Abstract convexity and augmented lagrangians.
\newblock SIAM Journal on Optimization \textbf{18}(2), 413--436 (2007)

\bibitem{cannarsa}
Cannarsa, P., Sinestrari, C.: Semiconcave Functions, Hamilton-Jacobi Equations,
  and Optimal Control, vol.~58.
\newblock Birkhauser Boston, MA (2004)

\bibitem{dolecki-k}
Dolecki, S., Kurcyusz, S.: On {$\Phi$}-convexity in extremal problems.
\newblock SIAM J. Control and Optimization \textbf{16}, 277--300 (1978)

\bibitem{Ekeland}
Ekeland, I., T{\'e}man, R.: Convex {A}nalysis and {V}ariational {P}roblems.
\newblock Society for Industrial and Applied Mathematics, Philadelphia, PA, USA
  (1999)

\bibitem{gomezvidal}
Fajardo, M.D., Vidal, J.: A comparison of alternative c-conjugate dual problems
  in infinite convex optimization.
\newblock Optimization \textbf{66}, 705--722 (2017).
\newblock \doi{10.1080/02331934.2017.1295046}

\bibitem{gasinski}
Gasiński, L., Papageorgiou, N.: Exercises in Analysis (2016).
\newblock \doi{10.1007/978-3-319-27817-9}

\bibitem{HareWarPol}
Hare, W., Poliquin, R.: The {Quadratic} {Sub}-{Lagrangian} of a prox-regular
  function.
\newblock Nonlinear Anal. \textbf{47}, 1117--1128 (2001).
\newblock \doi{10.1016/S0362-546X(01)00251-6}

\bibitem{Huang}
Huang, X.X., Yang, X.Q.: Further study on augmented lagrangian duality theory
  \textbf{31}(2), 193–--210 (2005).
\newblock \doi{10.1007/s10898-004-5695-7}

\bibitem{ioffe-rub}
Ioffe, A., Rubinov, A.: Abstract convexity and nonsmooth analysis: Global
  aspects.
\newblock Advances in Mathematical Economics \textbf{4}, 1--23 (2002)

\bibitem{Jey2007}
Jeyakumar, V., Rubinov, A.M., Wu, Z.Y.: Generalized {F}enchel's conjugation
  formulas and duality for abstract convex functions.
\newblock Journal of Optimization Theory and Applications \textbf{132}(3),
  441--458 (2007)

\bibitem{Kurcyusz}
Kurcyusz, S.: Some remarks on generalized lagrangians.
\newblock In: J.~Cea (ed.) Optimization Techniques Modeling and Optimization in
  the Service of Man Part 2, pp. 362--388. Springer Berlin Heidelberg, Berlin,
  Heidelberg (1976)

\bibitem{MARTINEZLEGAZ}
Martínez-Legaz, J.E., Volle, M.: Duality in {D}.{C}. programming: The case of
  several {D}.{C}. constraints.
\newblock Journal of Mathematical Analysis and Applications \textbf{237}(2),
  657 -- 671 (1999).
\newblock \doi{https://doi.org/10.1006/jmaa.1999.6496}

\bibitem{rolewicz}
Pallaschke, D., Rolewicz, S.: Foundations of Mathematical Optimization.
\newblock Kluwer Academic (1997)

\bibitem{penotrub}
Penot, J.P., Rubinov, A.M.: Multipliers and general {L}agrangians.
\newblock Optimization \textbf{54}(4-5), 443--467 (2005)

\bibitem{poli-rock96}
Poliquin, R.A., Rockafellar, R.T.: Prox-regular functions in variational
  analysis.
\newblock Trans. Amer. Math. Soc. \textbf{348}, 1805--1838 (1996)

\bibitem{RockWets98}
Rockafellar, R., Wets, R.J.B.: Variational Analysis.
\newblock Springer Verlag, Heidelberg, Berlin, New York (1998)

\bibitem{rocka}
Rockafellar, R.T.: Augmented lagrange multiplier functions and duality in
  nonconvex programming.
\newblock SIAM J. Control \textbf{12}(2), 268--285 (1974)

\bibitem{rol-glob}
Rolewicz, S.: On a globalization property.
\newblock Applicationes Mathematicae \textbf{22}(1), 69--73 (1993)

\bibitem{rolewicz1994}
Rolewicz, S.: Convex analysis without linearity.
\newblock Control and Cybernetics \textbf{23}, 247--256 (1994)

\bibitem{Rolewiczpara}
Rolewicz, S.: Paraconvex analysis.
\newblock Control and Cybernetics \textbf{34}, 951–--965 (2005)

\bibitem{rubbook}
Rubinov, A.M.: Abstract Convexity and Global Optimization.
\newblock Kluwer Academic,, Dordrecht (2000)

\bibitem{RubHua}
Rubinov, A.M., Huang, X.X., Yang, X.Q.: The zero duality gap property and lower
  semicontinuity of the perturbation function.
\newblock Mathematics of Operations Research \textbf{27}(4), 775--791 (2002)

\bibitem{SunLongLi}
Sun, X., Long, X.J., Li, M.: Some characterizations of duality for {DC}
  optimization with composite functions.
\newblock Optimization \textbf{66}(9), 1425--1443 (2017).
\newblock \doi{10.1080/02331934.2017.1338289}

\bibitem{Sun}
Sun, X.K., Guo, X.L., Zhang, Y.: Fenchel-{Lagrange} duality for {DC} programs
  with composite functions.
\newblock Journal of Nonlinear and Convex Analysis \textbf{16}, 1607--1618
  (2015)

\bibitem{Syga2018}
Syga, M.: Minimax theorems for extended real-valued abstract convex--concave
  functions.
\newblock Journal of Optimization Theory and Applications \textbf{176}(2),
  306--318 (2018).
\newblock \doi{10.1007/s10957-017-1210-4}

\bibitem{sygamoor}
Syga, M.: On global properties of lower semicontinuous quadratically minorized
  functions  (2020)

\bibitem{toland}
Toland, J.: Duality in nonconvex optimization.
\newblock Journal of Mathematical Analysis and Applications \textbf{66},
  399--415 (1978)

\bibitem{TuyDC}
Tuy, H.: {D}.{C}. {O}ptimization: {T}heory, {M}ethods and {A}lgorithms  (1995).
\newblock \doi{10.1007/978-1-4615-2025-2}

\bibitem{Vial}
Vial, J.P.: Strong and weak convexity of sets and functions.
\newblock Mathematics of Operations Research \textbf{8}(2), 231--259 (1983)

\bibitem{villani2003topics}
Villani, C.: Topics in Optimal Transportation.
\newblock Graduate studies in mathematics. American Mathematical Society
  (2003).
\newblock \urlprefix\url{https://books.google.pl/books?id=MyPjjgEACAAJ}

\bibitem{zalinescu2002}
Z\u{a}linescu, C.: Convex Analysis in General Vector Spaces.
\newblock World Scientific (2002)

\bibitem{zalinescu}
Zălinescu, C.: On the duality gap and gale's example in infinite-dimensional
  conic linear programming.
\newblock Journal of Mathematical Analysis and Applications \textbf{520}(1),
  126--868 (2023).
\newblock \doi{https://doi.org/10.1016/j.jmaa.2022.126868}

\end{thebibliography}
\end{document}